\documentclass[reqno,12pt]{article}

\usepackage{a4wide}
\usepackage{amsmath} 
\usepackage{amssymb}
\usepackage{amsthm}
\usepackage[utf8]{inputenc} 
\usepackage{graphicx} 
\usepackage[english, russian]{babel}

\numberwithin{equation}{section}

\newtheorem{Th}{Theorem}

\newtheorem{coro}{Corollary}
\newtheorem{Prop}{Proposition}
\newtheorem{lem}{Lemma}

\theoremstyle{remark}
\newtheorem{Remark}{Remark}

\newcommand{\N}{\mathbb{N}}
\newcommand{\Z}{\mathbb{Z}}
\newcommand{\Q}{\mathbb{Q}}
\newcommand{\R}{\mathbb{R}}
\renewcommand{\C}{\mathbb{C}}
\newcommand{\K}{\mathbb{K}}
\renewcommand{\L}{\mathbb{L}}

\newcommand{\Qbar}{\overline{\mathbb Q}}

\newcommand{\etoile}{^\star}
\newcommand{\Card}{{\rm Card}}

\newcommand{\unp}{\{1,\ldots,p\}}
\newcommand{\unq}{\{1,\ldots,q\}}

\newcommand{\unmu}{\{1,\ldots,\mu\}}

\newcommand{\unrhoplusell}{\{1,\ldots,\roro+\ell\}}

\newcommand{\eps}{\varepsilon}

\newcommand{\Li}{{\rm Li}}

\newcommand{\dd}{{\rm d}}
\newcommand{\tra}{ ^t \!}

\newcommand{\calY}{{\mathcal Y}}
\newcommand{\calN}{{\mathcal N}}
\newcommand{\OK}{{\mathcal O}_{\K}}

\newcommand{\ppP}{\kappa} 
\newcommand{\qqQ}{K}

\makeatletter
\newcommand*{\house}[1]{%
 \mathord{%
 \mathpalette\@house{#1}%
 }%
}
\newcommand*{\@house}[2]{%
 \dimen@=\fontdimen8 %
 \ifx#1\scriptscriptstyle\scriptscriptfont
 \else\ifx#1\scriptstyle\scriptfont
 \else\textfont\fi\fi
 3 %
 \sbox0{%
 $#1%
 \vrule width\dimen@\relax
 \overline{%
 \kern2\dimen@
 \begingroup % to keep changes of \dimen@ in #2 local
 #2%
 \endgroup
 \kern2\dimen@
 }%
 \vrule width\dimen@\relax
 \mathsurround=1.5\dimen@ % outside margin
 $%
 }%
 \ht0=\dimexpr\ht0-\dimen@\relax
 \dp0=\dimexpr\dp0+2\dimen@\relax
 \vbox{%
 \kern\dimen@ % reinsert previously removed space
 \copy0 %
 }%
}

\newcommand{\lun}{\ell_1}
\renewcommand{\Re}{{\rm Re}\, }
\renewcommand{\Im}{{\rm Im}\, }

\newcommand{\nun}{m}

\newcommand{\ev}{{\rm ev}}
\newcommand{\ord}{{\rm ord}}

\newcommand{\hz}{{\mathcal H}_0}
\newcommand{\hal}{{\mathcal H}_\alpha}

\newcommand{\calW}{{\cal W}}
\newcommand{\bfP}{{\bf P}}
\newcommand{\Pti}{\widetilde P}

\newcommand{\yti}{\widetilde y}

\newcommand{\unS}{\{1,\ldots,S\}}
\newcommand{\Yun}{Y^{\{2,u_0,s_0\}}}
\newcommand{\Yde}{Y^{\{ 1, \ppp \}}}
\newcommand{\Ydeun}{Y^{\{ 1,1 \}}}
\newcommand{\Yderoro}{Y^{\{ 1,\roro \}}}
\newcommand{\petityde}{y^{\{ 1, \ppp \}}}
\newcommand{\ppp}{p}
\newcommand{\lamun}{\lambda_{2,u_0,s_0}}
\newcommand{\lamunun}{\lambda_{2,u_0,1}}
\newcommand{\lamde}{\lambda_{1,\ppp}}

\newcommand{\taual}{\tau_\alpha}
\newcommand{\omal}{\vari_\alpha}
 \newcommand{\omun}{\vari_1}
 \newcommand{\omz}{\vari_0}
 \newcommand{\vari}{{\rm var}}

\newcommand{\roro}{\varrho}
\newcommand{\cstunnv}{c'_{1}}
\newcommand{\eneq}{\end{equation}}
\newcommand{\cstun}{c_1}
\newcommand{\cstnum}{c_2}
\newcommand{\csttr}{c_3}

\newcommand{\rk}{{\rm rk}}

\newcommand{\hol}{{\mathcal H}}

\newcommand{\calB}{{\mathcal{B}}}

\newcommand{\loga}{{\mathcal L}_\alpha}
\newcommand{\otr}{\varpi_1}

\newcommand{\AAA}{E}
\newcommand{\unqmro}{\{1,\ldots,q-\roro\}}
\newcommand{\Deltati}{\widetilde \Delta}
\newcommand{\cale}{\mathcal{E}}
\newcommand{\evalpha}{{\rm ev}_\alpha}
 
\newcommand{\calYti}{\widetilde{\mathcal Y}}

\newcommand{\Mti}{\widetilde{M}}
\newcommand{\cz}{\C[z]}
\newcommand{\capade}[1]{\kappa(#1)}
\newcommand{\capaf}{\capade{f}}
\newcommand{\capazero}{\kappa_0}
\newcommand{\dpr}{{\cal D}}
\newcommand{\Deltapr}{\Delta}

\begin{document}

\title{Linear independence of values of $G$-functions, II. Outside the disk of convergence}
\author{S. Fischler and T. Rivoal}
\date\today

 \selectlanguage{english}

\maketitle

\begin{abstract} Given any non-polynomial $G$-function $F(z)=\sum_{k=0}^\infty A_kz^k$ of radius of convergence $R$ and in the kernel a $G$-operator $L_F$, we consider the $G$-functions 
$F_n^{[s]}(z)=\sum_{k=0}^\infty \frac{A_k}{(k+n)^s}z^k$ for every integers $s\ge 0$ and $n\ge 1$. These functions can be analytically continued to a domain $\mathcal{D}_F$ star-shaped at $0$ and containing the disk $\{\vert z\vert <R\}$.

Fix any $\alpha \in\mathcal D_F \cap \Qbar^*$, not a singularity of $L_F$, and any number field $\K$ containing $\alpha$ and the $A_k$'s. Let $\Phi_{\alpha, S}$ be the $\K$-vector space spanned by the values $F_n^{[s]}(\alpha)$, $n\ge 1$ and $0\le s\le S$. We prove that $u_{\K,F}\log(S)\le \dim_\K(\Phi_{\alpha, S })\le v_FS$ for any $S$, for some constants $u_{\K,F}>0$ and $v_F>0$. 
This appears to be the first Diophantine result for values of $G$-functions evaluated outside their disk of convergence.

This theorem encompasses a previous result of the authors in [{\em Linear independence of values of G-functions}, 46 pages, J. Europ. Math. Soc., to appear], where $\alpha\in \Qbar^*$ was assumed to be such that $\vert \alpha\vert <R$. 
Its proof relies on an explicit construction of a Pad\'e approximation problem adapted to certain non-holomorphic functions associated to $F$, and it is quite different of that in the above mentioned paper. It makes use of results of Andr\'e, Chudnovsky and Katz on $G$-operators, of a   linear independence criterion \`a la Siegel over number fields, and of a far reaching generalization of Shidlovsky's lemma 
built upon the approach of Bertrand-Beukers  and Bertrand.
\end{abstract}

\section{Introduction}

Siegel~\cite{siegel} defined the class of $G$-functions to generalize the Diophantine properties of the logarithmic function, by opposition 
to the exponential function which he generalized with the class of $E$-functions. 
A series $F(z)=\sum_{k=0}^\infty A_k z^k \in \Qbar[[z]]$ is a $G$-function if the following three conditions are met (we fix an embedding of $\Qbar$ into $\mathbb C$):

1. There exists $C>0$ such that for any $\sigma \in \textup{Gal}(\Qbar/\mathbb Q)$ and any $k\ge 0$, $\vert \sigma(A_k)\vert \le C^{k+1}$.

2. Define $D_n$ as the smallest positive integer such that $D_nA_k$ is an algebraic integer for any $k\le n$. There exists 
$D>0$ such that for any $n\ge 0$, $D_n\le D^{n+1}$.

3. $F(z)$ is a solution of a linear differential equation with coefficients in $\Qbar(z)$ (holonomicity or $D$-finiteness).

\noindent A power series $\sum_{k=0}^\infty \frac{A_k}{k!} z^k$ is an $E$-function if, and only if, $\sum_{k=0}^\infty A_k z^k$ is a $G$-function. Algebraic functions over $\Qbar(z)$ and holomorphic at $z=0$ are $G$-functions. Transcendental $G$-functions include the polylogarithms 
$\Li_s(z)=\sum_{k=1}^\infty \frac{z^k}{k^s}$ for every integer $s\geq 1$, multiple polylogarithms, and 
the generalized hypergeometric series ${}_{p+1}F_p$ with rational parameters. The exponential function, Bessel's functions and 
the generalized hypergeometric series ${}_{p}F_p$ with rational parameters are examples of $E$-functions. 

The Diophantine theory of $E$-functions is well developped and complete in some sense; we refer the reader to \cite{Shidlovski} and the introduction of \cite{ar} for statements of classical transcendence and algebraic independence theorems for values of $E$-functions at algebraic points. A complete picture is not yet known for $G$-functions, for which transcendence results are very sparse, and in fact it is in general not even known whether interesting $G$-values like $\Li_5(1)=\zeta(5)$ are irrational or not. We refer the reader to the introduction of \cite{fns} for statements of classical Diophantine results on values of $G$-functions at algebraic points. The most general of them, due to Chudnovsky, proves in particular the irrationality of the value $F(\alpha)$ of any given $G$-function $F$ evaluated at an algebraic point $\alpha$ sufficiently close to the origin (and this depends on $F)$, and in particular very far from the circle of convergence of $F$. 

In \cite{fns}, we adopted a dual point of view: for any given $G$-function $F$ and any algebraic point $\alpha$ in the disk of convergence, we obtained a non-trivial estimate for the dimension of a certain vector space spanned by the values at $\alpha$ of a family of $G$-functions naturally associated to $F$. The goal of the present paper is to generalize the main result of \cite{fns} to values of $F$ {\em outside} the disk of convergence.

To begin with, we recall that any $G$-function is solution of a minimal non-zero differential equation over $\Qbar(z)$, of order $\mu$ say. A result due to Andr\'e, Chudnovsky and Katz (see \cite[p. 719]{andre}) asserts that this minimal equation has very specific properties, which will be used in the sequel: it is Fuchsian with rational exponents, and at any point $\alpha\in \Qbar \cup \{\infty\}$ it has a basis of solutions of the form $\big(f_1(z-\alpha), \ldots, f_\mu(z-\alpha)\big)\cdot (z-\alpha)^A$, where $A\in M_{\mu\times \mu}(\mathbb Q)$ is triangular superior, and the $f_j(z)\in \Qbar[z]$ are $G$-functions (if $\alpha=\infty$, $z-\alpha$ must be replaced by $1/z$). Such a non-zero minimal equation is called a $G$-operator.

\bigskip

To state our result, we introduce some notations. Starting from a $G$-function $F(z)=\sum_{k=0}^\infty A_k z^k$ with radius of convergence $R$, we define for any integers $n\ge 1$ and $s\ge 0$ the $G$-functions 
\begin{equation*}
F_n^{[s]}(z)=\sum_{k=0}^\infty \frac{A_k}{(k+n)^s}z^{k+n}
\end{equation*}
which all have $R$ as radius of convergence. Being an holonomic function, $F$ can be continued to a suitable cut plane. More precisely, let $\Sigma_F$ denote the set of finite singularities of $F$, which are either poles or branch points. For $\alpha\in \Sigma_F$, we define $\Deltapr_\alpha:=\alpha +e^{i\arg(\alpha)} \mathbb R^+$, the straight line ``from $\alpha$ to $\infty$'' whose direction goes through 0 but $0\notin \Deltapr_{\alpha}$. Then, $F$ and the $F_n^{[s]}$ can all be analytically continued to the domain $\mathcal{D}_F:=\mathbb C \setminus (\cup_{\alpha \in \Sigma_F} \Deltapr_\alpha)$, which is star-shaped at $0$ and contains the disk $\{\vert z\vert<R\}$.

\begin{Th}\label{thintroun}
Let $\K$ be a number field, and $F$ be a $G$-function with Taylor coefficients in $\K$, such that $F(z)\not\in\C[z]$. Let $L_F\in\K[z,\frac{\dd }{\dd z}]$ be a $G$-operator such that $L_F F=0$. Let $z_0\in\K\setminus\{0\}$; assume that $z_0$ is not a singularity of $L_F$ and that $z_0\in \mathcal{D}_F$.

Then for any $S$, the $\K$-vector space spanned by the numbers $F_n^{[s]}(z_0)$ with $n \geq 1$ and $0 \leq s \leq S$ has dimension over $\K$ at least $\frac{1+o(1)}{[\K:\Q]C(F)} \log(S)$, where $C(F)$ depends only on $F$; here $o(1)$ denotes a sequence that tends to 0 as $S\to\infty$.
\end{Th}
Note that Eq. \eqref{eqdevintro} below immediately implies that this dimension is bounded above by $\ell_1 S+\mu$ for every $S\ge 0$, where the quantities $\ell_1$ and $\mu$ will be defined later in the paper. Both can be computed from $L_F$ and are independent of $\K$. We observe that $\ell_1 S+\mu$ can in fact be replaced by $\ell_0 S+\mu$ (where $\ell_0$ is defined in the introduction of \cite{fns} and is $\le \ell_1$) and provided one uses the analytic continuation to $\mathcal{D}_F$ of Identity (5.2) of \cite{fns} instead of \eqref{eqdevintro}. 

\bigskip

The new point in Theorem \ref{thintroun} is that we allow $|z_0| \ge R$ (provided $z_0$ belongs to the star-shaped domain $\mathcal{D}_F$ and is not a singularity of $L_F$). All previous Diophantine results we are aware of for values of any given $G$-function $F$ deal with points $z_0$ such that the defining series $\sum_{k=0}^\infty A_k z^k$ of $F(z)$ is {\em convergent} at $z=z_0$. Our method enables us to also deal with points $z_0$ such that $\sum_{k=0}^\infty A_k z_0^k$ is {\em divergent}.
For $|z_0| < R$, Theorem \ref{thintroun} follows from 
\cite[Theorem~3]{fns}. 

It can also be observed from our proof that, in Theorem \ref{thintroun}, it is in fact not necessary to assume that the function $F(z)=\sum_{k=0}^\infty A_k z^k$ is exactly a $G$-function. Indeed, a similar result holds provided we have $A_{k}=\sum_{j=1}^J c_j A_{j,k}$ for some $c_j\in \mathbb C$ independent of $k$ and where the series $\sum_{k=0}^\infty A_{j,k} z^k \in \K[[z]]$ are $G$-functions all solutions  of 
a $G$-operator (playing the role of $L_F$).

Moreover, Corollary 1 in \cite{fns} generalizes as follows.
\begin{coro}\label{coro:tanintro1} Let us fix some rational numbers $a_1, \ldots, a_{p+1}$ and $b_1, \ldots, b_{p}$ such that
$a_i \not\in\Z\setminus\{1\}$ and $b_j\not\in-\N$ for any $i$, $j$. 
 Then for any 
 $z_0\in\Qbar^*$ such that $z_0\notin[1, +\infty)$, 
infinitely many of the hypergeometric values 
\begin{equation}\label{eq:tanintro2}
{}_{p+s+1}F_{p+s} \left[ 
\begin{matrix}
a_1, a_2, \ldots, a_{p+1}, 1, \ldots, 1
\\
b_1, b_2, \ldots, b_p, 2, \ldots, 2
\end{matrix}
; z_0
\right], 
\quad s\ge 0
\end{equation}
are linearly independent over $\mathbb Q(z_0)$.
\end{coro}
To deduce Corollary \ref{coro:tanintro1} from Theorem \ref{thintroun}, one proceeds as in \cite{fns} by applying the analytic continuation to $\mathcal{D}_F$ of Identity (5.2) of \cite{fns}, with the same parameters as in \cite{fns}. That Identity (5.2) is formally similar to \eqref{eqdevintro}; the main difference is that for some values of the parameters we cannot take $\lun = 1$ in  \eqref{eqdevintro}. (See Remark~\ref{remcoro} in \S \ref{subseclienfns}  for related comments).

\bigskip

Of course, a more precise version of Corollary \ref{coro:tanintro1} holds: the dimension over $\mathbb Q(z_0)$ of the $\mathbb Q(z_0)$-vector space generated by the numbers in \eqref{eq:tanintro2} for $0\le s \le S$ is larger than $C \log(S)$ for some constant $C$ that depends on the $a$'s and $b$'s, and on $[\Q(z_0):\Q]$. 
We recall that if $\vert z_0\vert <1$, the hypergeometric numbers in \eqref{eq:tanintro2} are equal, by definition, to the convergent series
$$
\sum_{k=0}^\infty \frac{(a_1)_k(a_2)_k\cdots (a_{p+1})_k}{(1)_k(b_1)_k\cdots (b_p)_k} \frac{z_0 ^k}{(k+1)^s}
$$
with $(a)_k = a(a+1)\ldots(a+k-1)$.
If $\vert z_0\vert > 1$, the numbers \eqref{eq:tanintro2} are defined by the analytic continuation of the hypergeometric series to the cut plane 
 $\mathbb C\setminus[1, +\infty)$. These numbers can in fact be also expressed as finite linear combinations of values of {\em convergent} hypergeometric series with rational parameters, where the coefficients of the combinations are in $\Qbar[\log(\Qbar^*)]$. This is a consequence of the connection formulas between the solutions at $0$ and $\infty$ of the generalized hypergeometric differential equation (see \cite[\S 3]{norlund}). For the polylogarithms analytically continued to $\mathbb C\setminus[1, +\infty)$ with $-\pi<\arg(1-z)<\pi$, the connection formula is as follows: for any integer $s\ge1$, we have
\begin{equation}\label{eq:tanpoly}
\Li_s(z)=(-1)^{s+1} \Li_s\Big(\frac{1}{z}\Big)-\frac{(2i\pi)^s}{s!}B_s\Big(\frac{\log(z)}{2i\pi}\Big), \quad z\notin [0,+\infty)
\end{equation} 
where $\log(z)=\ln\vert z\vert+i\arg(z)$, $\arg(z)\in(0,2\pi)$, and 
$B_s$ is the $s$-th Bernoulli polynomial (see \cite[\S 1.3]{oesterle}). In particular, $\Li_2(z)=-\Li_2(\frac1z)+\frac{\pi^2}{3}-\frac{1}{2}\log(z)^2+i\pi \log(z)$ in these conditions. The following specialization of Corollary~\ref{coro:tanintro1} (and its implicit dimension lower bound) is interesting because it improves on the best known result due to Marcovecchio \cite{Marcovecchio}, which is restricted to the case $0<\vert z_0\vert<1$.
\begin{coro} \label{coro:tanintro2} For any 
 $z_0\in\Qbar^*$ such that $z_0\notin[1, +\infty)$, 
infinitely many of the values $\Li_s(z_0)$, $s\ge1$, are linearly independent over $\mathbb Q(z_0)$.
\end{coro}
By \eqref{eq:tanpoly}, when $\vert z_0\vert> 1$ and $z_0\notin [1,\infty)$, $\Li_s(z_0)$ is a $\mathbb Q$-linear combination of $\Li_s(1/z_0)$ (defined by the convergent series $\sum_{k=1}^\infty z_0^{-k}/k^s$) and powers of $\log(z_0)$ and $2i\pi$.

\medskip

The proof of Theorem \ref{thintroun} is not a simple generalization of the corresponding theorem in \cite{fns}, even though the strategy is similar at the beginning. There are important difficulties.
As in \cite{fns}, we first consider the auxiliary function
\begin{equation}\label{eqJFintro}
J_{F}(z) = n!^{S-r}\sum_{k=0}^\infty \frac{k(k-1)\cdots (k-rn+1)}{(k+1)^S(k+2)^S\cdots (k+n+1)^S} \,A_k \,z^{n+k+1}
\end{equation}
where {\em a priori} $\vert z\vert < R$, $r$ and $n$ are integer parameters such that $r\leq S $ and eventually $n\to +\infty$. 
We have proved in \cite{fns} that it can be expressed as a linear combination of the functions $F_n^{[s]}$ and $(z \frac{\dd }{\dd z})^u F$ with coefficients in $\K[z]$. Here, we prove similarly that 
\begin{equation}\label{eqdevintro}
J_F(z)=\sum_{u=1}^{\lun}\sum_{s=1}^S P_{u,s,n}(z) F_{u}^{[s]}(z)+ \sum_{u=0}^{\mu-1} \widetilde{P}_{u,n}(z) (\theta^uF)(z) 
\end{equation}
where, in particular, the parameter $\ell_1$ is defined in \S\ref{subsecnotationspade}.  
For $|z| < R$, in \cite{fns} we applied the saddle point method to estimate $ \vert J_F(z)\vert $ precisely; this enabled us to apply a generalization of Nesterenko's linear independence criterion \cite{nest} and deduce Theorem \ref{thintroun} in this case. When $|z|\geq R$, the first thing we would like to point out is that $ \vert J_F(z)\vert $ is still small enough to prove linear independence results; this could seem surprising since the series \eqref{eqJFintro} is divergent if $|z| > R$. Instead, we use an integral representation of $J_f(z)$ to bound $ \vert J_F(z)\vert $ from above (see \S \ref{subsec:ubJF}). However we did not try to bound $ \vert J_F(z)\vert $ from below (and don't even know if this would be possible) because the saddle point method would not be easy at all to generalize to deal with the analytic continuation of $J_F(z)$; indeed, there are immediate problems when one tries to extend it first to the case $\vert z\vert = R$ (for instance, the number of certain critical points suddenly drops when one shifts from the case $\vert z\vert <R$ to the case $\vert z\vert=R$). Therefore we cannot follow the original approach of \cite{BR, rivoalcras} based on Nesterenko's linear independence criterion. This is the reason why we follow the strategy of \cite{SFcaract}: we first prove that the polynomial coefficients of \eqref{eqdevintro} are solution of a Pad\'e approximation problem (namely Theorem \ref{thpade} stated in \S \ref{subsecnotationspade}). This enables us to apply a general version of Shidlovsky's lemma; since some solutions of the corresponding differential system may have identically zero Pad\'e remainders, the version of Shidlovsky's lemma proved in \cite{SFcaract} (based upon differential Galois theory, following the approach of Bertrand-Beukers \cite{BB} and Bertrand \cite{DBShid}) does not apply to our setting. We generalize it in \S \ref{secshid}, and we hope this result will have other applications. To conclude the proof we apply a linear 
independence criterion in the style of Siegel. 

\bigskip

The structure of this paper is as follows. In \S \ref{sec2} we focus on the two main tools in our approach. First we extend in \S \ref{subsecdeffjs} the definition of $f_j^{[s]}$, given in \cite{fns} when $f$ is holomorphic at 0, to any function in the Nilsson class with rational exponents at 0; then we study in \S \ref{subsecvar} the variation around a singularity of a solution of a Fuchsian operator, which generalizes the weight of polylogarithms. Section \ref{secpadestatement} is devoted to the statement of Theorem \ref{thpade}, namely the Pad\'e approximation problem satisfied by the polynomial coefficients of \eqref{eqdevintro}. We also make comments on this problem, especially on the fact that vanishing conditions involve non-holomorphic functions. Then we prove Theorem \ref{thpade} in \S \ref{secpadepreuve}, building upon the construction of \cite{fns}. We move in \S \ref{secshid} to the general version of Shidlovsky's lemma (namely Theorem \ref{thshid}) that we state and prove. We apply it in \S \ref{secapplishid} to the Pad\'e approximation problem of Theorem \ref{thpade}: this enables us to construct linearly independent linear forms (see Proposition \ref{propshid}). For the convenience of the reader, we state and prove in \S \ref{sec:critere} a (rather classical) version of Siegel's linear independence criterion. Section \ref{sec:estimates} contains the estimates needed in \S \ref{secfin} to conclude the proof of Theorem \ref{thintroun}.

\section{Properties of $f_j^{[s]}$ and variation around a singularity} \label{sec2}

In this section we focus on the two main tools in our approach. First we extend the definition of $f_j^{[s]}$, given in \cite{fns} when $f$ is holomorphic at 0, to any function in the Nilsson class with rational exponents at 0; we give both an inductive definition and an explicit formula. Then we study in \S \ref{subsecvar} the variation around a singularity of a solution of a Fuchsian operator, which generalizes the weight of polylogarithms.

\bigskip

Throughout this section we fix a simply connected dense open subset $\dpr$ of $\C$, with $0\not\in\dpr$; all functions we consider will be holomorphic on $\dpr$. In all applications we have in mind, $\dpr$ will be the subset defined after Lemma \ref{lemopdiff} in \S \ref{subsecnotationspade}. We also fix a determination of $\log (z)$, holomorphic on $\dpr$; then $z^k =\exp\big(k \log (z)\big)$ is well-defined for $z\in\dpr$ and $k\in\Q$.

\subsection{General definition of $f_j^{[s]}$} \label{subsecdeffjs}

Since $G$-operators are Fuchsian with rational exponents, all the functions we consider in this paper have local expansions at 0 of the form 
\begin{equation} \label{eqNil}
f(z) = \sum_{k\in \Q \atop k\geq \capaf} \sum_{i=0}^{e} a_{k,i} z^k \log (z)^i
\end{equation}
with $a_{k,i}\in \mathbb C$. 
We denote by $\calN$ the set of all functions holomorphic on $\dpr$ that have such an expansion at 0 (i.e., belong to the Nilsson class with rational exponents at 0) -- recall that $\dpr$ is a simply connected dense open subset of $\C\setminus\{0\}$ fixed in this section. It is important to observe that $\calN$ is a $\mathbb C$-differential algebra stable by primitivation.
(Restricting to rational exponents is not really needed, but we do it because it contains all applications we have in mind.)

\bigskip

The notation $f_j^{[s]}$ has been defined in \cite{fns} for $f$ holomorphic at 0; let us extend this definition to any $f\in \calN$.
With this aim in view, we first let $\ev : \calN \to\C$ denote the ``regularized'' evaluation at 0: if $f(z) = \sum_{k,i } a_{k,i} z^k \log (z)^i $ around $z=0$ then $\ev(f) = a_{0,0}$ by definition. Then for any $f\in\calN $, consider any of its primitives $\mathcal{P}(f)$: it is a basic fact that $\mathcal{P}(f)\in\calN$ as well. We will denote by 
$\int^z f(x) dx$ (with no lower bound in the integral) the unique primitive $g$ of $f$ such that $\ev(g) = 0$. Of course we have $g(z) = \int_0^z f(x) dx$ if this integral is convergent. Moreover, $\int^z$ is a $\mathbb C$-linear operator acting on $\calN$. 

Given $f\in\calN$ and $j\geq 1$, we let 
\begin{equation} \label{eqdefun}
f_j(z) =f_j^{[1]}(z) = \int^z x^{j-1} f(x) dx
\end{equation}
and we define recursively $f_j^{[s]}$ for $s\geq 2$ by 
\begin{equation} \label{eqdefde}
 f_j^{[s+1]}(z) = \int ^z \frac1{x} f_j^{[s]}(x) dx . 
\end{equation}
This defines $f_j^{[s]}\in\calN $ for any $j,s\geq 1$: the function $ f_j^{[s]}$ is holomorphic on $\dpr$, and its local expansion at 0 is of the form \eqref{eqNil}. 

\bigskip

Let us focus on a few special cases. For any $f\in\calN$ there exists $\capaf\in\Q$ such that 
$$f(z) = \sum_{k\in \Q \atop k\geq \capaf} \sum_{i=0}^{e} a_{k,i} z^k \log (z)^i$$
around $z=0$.
If $j\geq 1$ is such that $j > -\capaf$ then, for any $s\geq 1$, in the process of defining $f_j^{[s]}$, all integrals are convergent; therefore in Eqns. \eqref{eqdefun} and \eqref{eqdefde} 
we may replace $\int^z$ with $\int_0^z$. This ``convergent'' case is the most important one. If $f$ is holomorphic at 0 then we may take $\capaf=0$ and this situation holds for any $j\geq 1$. More generally, if the local expansion of $f$ at 0 reads 
$$f(z) = \sum_{k\in \Q \atop k\geq \capaf} a_{k} z^k $$
with $a_k=0$ for any negative integer $k$, then the above definition yields 
$$
f_j^{[s]}(z) = \sum_{k\in \Q \atop k\geq \capaf} \frac{a_k}{(k+j)^s}z^{k+j}
$$
for any $j,s\geq 1$, with $|z|$ small enough. In this sum, the terms corresponding to $k\in\Z$, $k<0$, have to be omitted; this is harmless since we have assumed $a_k=0$ in this case. Of course, if $f$ is holomorphic at 0 then we recover the definition of $f_j^{[s]}$ given in \cite{fns}. 
The situation where $a_k\neq 0$ for some negative integer $k$ is more subtle. For instance, if $f(z) = z^k$ with $k\in \Z$, $k<0$, then for any $j,s\geq 1$ we have 
$$
f_j^{[s]}(z) = \begin{cases} 
\frac{1}{(k+j)^s}z^{k+j} \; \mbox{ if } j \neq -k,
\\
\frac{1}{s!}\log(z)^s \; \mbox{ if } j = -k.
\end{cases}
$$
More generally, the following lemma provides an explicit formula to compute $f_j^{[s]}$. It will be used to prove Lemma \ref{lemexplicite} in \S\ref{ssec:expcompJf3110}.

\begin{lem} \label{lemexplicitefns}
Let $f \in \mathcal{N}$ be such that 
\begin{equation}\label{eq:f3010}
f(z) = \sum_{k\in \Q \atop k\geq \capaf} \sum_{i=0}^{e} a_{k,i} z^k \log (z)^i
\end{equation}
for $|z|$ small enough. 
Then for any $j\ge 1,s\geq 1$ and any $z\in\dpr$ we have
\begin{equation}\label{eq:fjs3010}
f_j^{[s]}(z) = \sum_{i=0}^e a_{-j, i} i! \frac{\log (z)^{s+i}}{(s+i)! } 
+ \sum_{k\in \Q\setminus\{-j\} \atop k\geq \capaf} \sum_{\lambda=0}^{e} z^{k+j} \frac{\log (z)^\lambda}{\lambda!} \sum_{i=\lambda}^e a_{k,i} (-1)^{i-\lambda} \frac{i! \binom{s-1+i-\lambda}{s-1} }{(k+j)^{s+i-\lambda}} .
\end{equation}
\end{lem}

\begin{Remark} We may define $f_j^{[s]}$ also for $s=0$, by letting $f_j^{[0]}(z) = z^j f(z)$. Then Eq. \eqref{eq:fjs3010} holds also in this case, provided we write $\binom{s-1+i-\lambda}{s-1} = \frac{s}{s+i-\lambda} \binom{s+i-\lambda}{s}$ and consider that for $s=0$, this is equal to 1 if $i = \lambda$, and to 0 otherwise. 
\end{Remark}

\begin{proof} We shall freely use the following elementary primitivation formulas: for every integers $m,n \in \mathbb Z$ such that $m\neq -1$ and $n\ge 0$, we have on $\dpr$
\begin{equation*}
\mathcal{P}\big(z^m\log(z)^n\big) 
=z^{m+1}\sum_{\ell=0}^n \binom{n}{\ell} \frac{(-1)^{n-\ell}(n-\ell)!}{(m+1)^{n-\ell+1}}\log(z)^\ell +c, \quad \mathcal{P}\bigg(\frac{\log(z)^n}{z}\bigg) = \frac{\log(z)^{n+1}}{n+1}+c,
\end{equation*}
where, in both cases, the constant $c\in \mathbb C$ is arbitrary. In particular, $\int^z x^m\log(x)^n \,dx$ and $\int^z \log(x)^n/x \,dx$ are the primitives for which $c=0$ in these formulas.

We will also need the following combinatorial identity: for every integers $i,j,s\ge 0$, 
\begin{equation}\label{eq:combi3010}
\sum_{\lambda=j}^i \binom{s+i-\lambda}{s} = \binom{s+1+i-j}{s+1}.
\end{equation}
It is readily proved by noticing that $\binom{s+i-\lambda}{s}=\binom{s+1+i-\lambda}{s+1}-\binom{s+i-\lambda}{s+1}$, which transforms the sum on the left-hand side of \eqref{eq:combi3010} into a telescoping one (with the usual convention that $\binom{a}{b}=0$ if $a<b$).

We now prove Eq. \eqref{eq:fjs3010} by induction on $s\ge 1$. Let us prove the case $s=1$. For any integer $j\ge 1$ and any $z\in \dpr$, we have
\begin{align*}
f_j^{[1]}(z)&= \int^z x^{j-1}\bigg(\sum_{k\in \Q \atop k\geq \capaf} \sum_{i=0}^{e} a_{k,i} x^{k} \log (x)^i\bigg) dx
\\
&=\sum_{i=0}^{e} a_{-j,i} \int^z \frac{\log(x)^i}{x} dx + \sum_{k\in \Q\setminus\{-j\} \atop k\geq \capaf} \sum_{i=0}^{e} a_{k,i} \int^z x^{k+j-1} \log(x)^i dx
\\
&=\sum_{i=0}^{e} \frac{a_{-j,i}}{i+1}\log(z)^{i+1}+
\sum_{k\in \Q\setminus\{-j\} \atop k\geq \capaf} z^{k+j}\sum_{i=0}^{e} a_{k,i}\sum_{\lambda=0}^{i} \binom{i}{\lambda} \frac{(-1)^{i-\lambda}(i-\lambda)!}{(k+j)^{i-\lambda+1}}\log(z)^\lambda 
\\
&=\sum_{i=0}^{e} a_{-j,i}i!\frac{\log(z)^{i+1}}{(i+1)!}+
\sum_{k\in \Q\setminus\{-j\} \atop k\geq \capaf} z^{k+j} \sum_{\lambda=0}^{e} \frac{\log(z)^\lambda}{\lambda !}\sum_{i=\lambda}^{e} a_{k,i} \frac{(-1)^{i-\lambda}i!}{(k+j)^{i-\lambda+1}}
\end{align*}
so that Eq.~\eqref{eq:fjs3010} is proved in the case $s=1$.

Let us now assume that Eq.~\eqref{eq:fjs3010} is proved for $f_j^{[s]}$ and let us prove it for $f_j^{[s+1]}$. For simplicity, we define 
$$
C_{\lambda}(u,s):=\sum_{i=\lambda}^{e} a_{k,i} \frac{(-1)^{i-\lambda}i!}{u^{s+i-\lambda}}\binom{s-1+i-\lambda}{s-1}
$$
so that \eqref{eq:fjs3010} becomes 
\begin{equation}\label{eq:fjs13010}
f_j^{[s]}(z)= \sum_{i=0}^{e} \frac{a_{-j,i}i!}{(s+i)!}\log(z)^{s+i}+\sum_{k\in \Q\setminus\{-j\} \atop k\geq \capaf} z^{k+j} \sum_{\lambda=0}^{e} C_{\lambda}(k+j,s) \frac{\log(z)^\lambda}{\lambda!}.
\end{equation}
For any integer $j\ge 1$ and any $z\in \dpr$, we have:
\begin{align*}
f_j^{[s+1]}(z)&=\int^z \frac{1}{x} f_j^{[s]}(x) dx
\\
&=\int^z \frac{1}{x}\bigg(\sum_{i=0}^{e} \frac{a_{-j,i}i!}{(s+i)!}\log(x)^{s+i}\bigg) dx
\\
&\qquad \qquad+\int^z \frac{1}{x}\bigg(\sum_{k\in \Q\setminus\{-j\} \atop k\geq \capaf} x^{k+j} \sum_{\lambda=0}^{e} C_{\lambda}(k+j,s)\frac{\log(x)^\lambda}{\lambda!}\bigg) dx 
\\
&=\sum_{i=0}^{e} \frac{a_{-j,i}i!}{(s+i)!(s+i+1)}\log(z)^{s+i+1}
\\
&\qquad \qquad +\sum_{k\in \Q\setminus\{-j\} \atop k\geq \capaf} \sum_{\lambda=0}^{e} C_{\lambda}(k+j,s)\sum_{\ell=0}^{\lambda}
\frac{(-1)^{\lambda-\ell}\binom{\lambda}{\ell}(\lambda-\ell)!}{(k+j)^{\lambda-\ell+1}\lambda!}z^{k+j}\log(z)^\ell 
\\
&=\sum_{i=0}^{e} \frac{a_{-j,i}i!}{(s+i+1)!}\log(z)^{s+i+1}
\\
&\qquad \qquad +\sum_{k\in \Q\setminus\{-j\} \atop k\geq \capaf} z^{k+j} \sum_{\ell=0}^{e} \frac{\log(z)^\ell}{\ell!} \sum_{\lambda=\ell}^{e} C_{\lambda}(k+j,s)
\frac{(-1)^{\lambda-\ell}}{(k+j)^{\lambda-\ell+1}} .
\end{align*}
Now it remains to simplify the inner sum over $\lambda$. We have 
\begin{align}
\sum_{\lambda=\ell}^{e} C_{\lambda}(k+j,s)
\frac{(-1)^{\lambda-\ell}}{(k+j)^{\lambda-\ell+1}}&=\sum_{\lambda=\ell}^{e}\sum_{i=\lambda}^{e} a_{k,i} \frac{(-1)^{i-\lambda}i!\binom{s-1+i-\lambda}{s-1}}{(k+j)^{s+i-\lambda}} \frac{(-1)^{\lambda-\ell}}{(k+j)^{\lambda-\ell+1}} \notag
\\
&=\sum_{i=\ell}^e a_{k,i}\frac{(-1)^{i-\ell}i!}{(k+j)^{s+1+i-\ell}}\sum_{\lambda=\ell}^i \binom{s-1+i-\lambda}{s-1} \label{eq:innersum3010}
\\
&=\sum_{i=\ell}^e a_{k,i}\frac{(-1)^{i-\ell}i!}{(k+j)^{s+1+i-\ell}} \binom{s+i-\ell}{s} \notag
\\
&=C_\ell(k+j,s+1), \notag
\end{align}
where we have used Identity \eqref{eq:combi3010} to compute the inner sum in Eq. \eqref{eq:innersum3010}. Therefore, 
\begin{equation*}
f_j^{[s+1]}(z)=\sum_{i=0}^{e} \frac{a_{-j,i}i!}{(s+i+1)!}\log(z)^{s+i+1}
+\sum_{k\in \Q\setminus\{-j\} \atop k\geq \capaf} z^{k+j} \sum_{\ell=0}^{e} C_\ell(k+j,s+1)\frac{\log(z)^\ell}{\ell!}.
\end{equation*}
This is nothing but Eq. \eqref{eq:fjs13010} with $s$ replaced by $s+1$, which completes the proof of the induction and that of Lemma~\ref{lemexplicitefns}. 
\end{proof}

\subsection{Variation around a singularity} \label{subsecvar}

Recall that $\dpr$ is a simply connected dense open subset of $\C$, with $0\not\in\dpr$, fixed in this section. Let $L$ be a Fuchsian operator with rational exponents at 0 (for instance a $G$-operator); assume that $\dpr$ does not contain any singularity of $L$. Let $f$ be a function holomorphic on $\dpr$, such that $Lf=0$.

\bigskip

For any $\alpha\in\C$ and any $z\in\dpr$ we denote by $\taual f(z)$ the value obtained by analytic continuation of $f$ along the following path: starting from $z$, go very close to $\alpha$ (while remaining in $\dpr$), then do a small loop going around $\alpha$ once (in the positive direction), and at last come back to $z$ (remaining again in $\dpr$). Observe that if $\alpha\not\in\dpr$ (for instance if $\alpha$ is a singularity of the Fuchsian operator we consider), then the loop around $\alpha$ does not remain in $\dpr$ so that $\taual f(z)$ will be distinct from $f(z) $ in general. This process defines a function $\taual f$ holomorphic on $\dpr$. We also let 
$$\omal f(z) = \taual f(z) - f(z)$$
denote the variation of $f$ around $\alpha$; this function is also holomorphic on $\dpr$. For instance, $\omun \log(1-z) = 2i \pi$ (with $\dpr = \C \setminus [0, +\infty)$), and $\omal f(z)=0$ is $f$ is meromorphic at $\alpha$. An important property of analytic continuation is that monodromy (in particular variation) commutes with differentiation: we have $(\omal f)' = \omal (f')$.

\begin{lem}\label{lempoids}
Assume that the $k$-th derivative $f^{(k)}(z)$ has a finite limit as $z\to \alpha$, for any $k \in \{0,\ldots,K\}$. Then we have
$$\omal f(z) = o((z-\alpha)^K)\mbox{ as } z\to\alpha.$$
\end{lem}

\begin{Remark} In the case of the polylogarithms, by \cite[p. 53, Proposition 1]{oesterle}, we have 
$$
\textup{var}_1 \Li_s(z)=-(2i\pi) \frac{\log(z)^{s-1}}{(s-1)!}, \quad s\ge 1
$$ 
We observe that $\textup{var}_1 \Li_s$ is essentially the weight of $\Li_s$ (see for instance \cite[Lemma 4.1]{Pise}). The case $K=0$ of Lemma \ref{lempoids} corresponds to the remark before Lemma 4.2 of \cite{Pise}: if $f$ has a finite limit at $\alpha$ then $\omal f(\alpha) = 0$.
\end{Remark}

\bigskip 

\begin{proof}[Proof of Lemma \ref{lempoids}] For any $k \in \{0,\ldots,K\}$ we have $(\omal f)^{(k)} = \omal ( f ^{(k)} ) = \taual (f^{(k)} ) - f^{(k)} $. Since $f^{(k)} (z) $ has a finite limit as $z\to\alpha$, $\taual (f^{(k)} )(z)$ has the same limit since the path used for analytic continuation can remain very close to $\alpha$ as $z\to\alpha$. Therefore $\lim_{z\to\alpha} (\omal f)^{(k)} (z) = 0$ for any $k \in \{0,\ldots,K\}$. Now $f$ is annihilated by a Fuchsian operator, so that the same holds for $\taual f$ and $\omal f$: there exist $c \in\C$, $r\in\Q$ and $j\in\N$ such that $\omal f(z) \sim c (z-\alpha)^r (\log(z-\alpha))^j$ as $z\to\alpha$. Since all derivatives up to order $K$ vanish at $\alpha$, we have $r>K$. This concludes the proof of Lemma \ref{lempoids}.
 \end{proof}
 
\bigskip

\begin{lem}\label{lempoidsalpha}
For any $\alpha\in\C\setminus\{0\}$ and any $u,s \geq 1$ there exists $P_{\alpha,u,s,f} \in\C[X]$ of degree at most $s$ such that, for any $z\in\dpr$:
$$\omal ( f_u^{[s]}) (z) = (\omal f)_u^{[s]} (z ) + P_{\alpha,u,s,f} ( \log z).$$
\end{lem}

\begin{proof}
 If $s\geq 2$, the derivative of the left hand side is 
$$\omal \Big( \frac{d}{dz} f_u^{[s]} \Big) = \omal \Big( \frac1z f_u^{[s-1]} \Big) = \frac1z \omal ( f_u^{[s-1]} ).$$
Computing the derivative of the right hand side shows that Lemma \ref{lempoidsalpha} holds for $s$ if it does for $s-1$, by taking for $P_{\alpha,u,s,f}(0)$ the appropriate constant of integration. The same argument proves Lemma \ref{lempoidsalpha} for $s=1$.
 \end{proof}

\bigskip

\begin{lem}\label{lempoidscommute}
Let $u\geq 1$ be such that $u > -\capaf$, where $\capaf$ is given by the local expansion \eqref{eqNil} of $f$ at 0. 
 Then for any $s\geq 1$ we have
$$\omz (f_u^{[s]}) = (\omz f)_u^{[s]} .$$
\end{lem}

\begin{proof}
 Since $u > -\capaf$, in the process of defining $f_u^{[s]}$ all integrals are convergent; the same holds for $
 (\omz f)_u^{[s]}$ because we may choose $\capade{\omz f} = \capaf$. Therefore we have $f_u^{[s]} (0) = (\omz f)_u^{[s]} (0)$; then Lemma \ref{lempoids} (with $K=0$) yields $(\omz (f_u^{[s]}))(0)=0$. Moreover $\omz$ commutes with differentiation, so that $\omz (f_u^{[s]})$ and $ (\omz f)_u^{[s]} $ have the same derivative (assuming inductively that either $s=1$ or Lemma \ref{lempoidscommute} holds with $s-1$). This concludes the proof of Lemma \ref{lempoidscommute}.
 \end{proof}

\section{Pad\'e approximation problem: statement and comments} \label{secpadestatement}
 
In this section we state and prove a Pad\'e approximation problem satisfied by the polynomials constructed in \cite{fns}, namely Theorem \ref{thpade}. With this aim in mind we recall in \S \ref{subsecnotationspade} the notation and the output of the construction of \cite{fns}. An interesting feature of this Pad\'e approximation problem is that non-holomorphic functions appear in the vanishing conditions; we comment this property in \S \ref{subsecnonholopade}, and explain how to count the linear equations equivalent to such conditions. This allows us to show in \S \ref{subseccompte} that this Pad\'e approximation problem has essentially as many equations as its number of unknowns.

\subsection{Notation and statement} \label{subsecnotationspade} 
 
 Let $F$ be a $G$-function, fixed throughout this paper, with Taylor coefficients at 0 in a number field $\K$. Let $L_F \in \K[z, \frac{\dd}{\dd z}]$ be a $G$-operator such that $L_FF=0$. To simplify the exposition and avoid dealing with microsolutions in \S \ref{subsecannulationsingularites} (see for instance \cite{fuchsien}), instead of $L_F$ we shall use the differential operator $L$ provided by the following lemma, and it is also a $G$-operator.

\begin{lem} \label{lemopdiff}
Let $L_0$ be a $G$-operator. There exists a non-negative integer $N$ such that, upon letting $L = (\frac{d}{dz})^N \circ L_F$:
\begin{equation} \label{eqopdiff}
\mbox{for any $f$ such that $Lf\in\C[z]$ there exists $P\in\C[z] $ such that $L(f+P)=0$.}
\end{equation}
\end{lem}

\begin{proof} The quotient vector space $L_F^{-1}(\C[z])/\C[z]$ is finite-dimensional; it has a basis $(f_1\bmod \C[z], \ldots, f_p\bmod \C[z])$. For any $i\in\unp$ we have $L_F f_i \in\C[z]$; choose $N$ such that $\deg L_0 f_i < N$ for any $i$, and let $L = (\frac{d}{dz})^N \circ L_F$. For any $f$ such that $L f \in \C[z]$, we have $L_F f \in\C[z]$ so that $f = \lambda_1 f_1 + \ldots + \lambda_p f_p - P$ for some $\lambda_1,\ldots,\lambda_p \in \C$ and $P \in \C[z]$; then $f+P = \lambda_1 f_1 + \ldots + \lambda_p f_p \in L_F^{-1}(\cz_{<N}) = \ker L$. This concludes the proof of the lemma.\end{proof}

\bigskip

Let us recall the notation of \cite{fns} (see \S \ref{subseclienfns} for details). We let $\theta = z\frac{\dd}{\dd z}$, and denote by $\mu$ the order of $L$, by $\Sigma$ the set of all finite singularities of $L$, and by $\ker L$ the space of solutions of $L$. We also define
$$\dpr = \C \setminus \Big( \bigcup_{\alpha\in\Sigma\cup\{0\}} \Deltapr_{\alpha}\Big)$$
where $ \Deltapr_{\alpha} = \{t\alpha, \, t\in\R, \, t\geq1\}$ for $\alpha\in\Sigma\setminus\{0\}$ and $ \Deltapr_0 $ is a fixed closed half-line starting at 0 such that $ \Deltapr_0 \cap \Deltapr_\alpha = \emptyset$ for any $\alpha\in\Sigma\setminus\{0\}$.
We have $\dpr\subset \mathcal{D}_F$ where $\mathcal{D}_F$ is defined in the introduction; notice that we consider here $\Deltapr_\alpha$ for any $\alpha\in \Sigma\cup\{0\}$, not only for singularities of $F$.
 Then $\dpr$ is a simply connected dense open subset of $\C$; we have $0\notin\dpr$, and $\dpr$ does not contain any singularity of $L$. Therefore the notation and results of \S \ref{sec2} apply to $\dpr$. All elements of $\ker L$ will be considered as holomorphic functions on $\dpr$, and will often be expanded around singularities of $L$ (recall that $L$ is Fuchsian with rational exponents at all singularities).

\bigskip

We denote by $\hz$ the space of functions holomorphic at 0, and let 
\begin{equation}\label{eqdefell}
\ell := \dim \Big( \frac{ \hz\cap\ker L }{\cz \cap\ker L}\Big).
\end{equation}
We remark that this definition is equivalent to that of \cite{fns}. Indeed, with the latter, holomorphic solutions of the differential equation $Ly=0$ are given by $y(z)=\sum_{k=0}^\infty a_k z^k$ where the sequence $(a_k)$ is determined (for $k$ large enough) by a linear recurrence relation of order $\ell$ (see \cite[Lemma 1 and Step 1 of the proof of Lemma 2]{fns}). This means that $\ell = \dim ( \frac{ \hz\cap L^{-1}\cz }{\cz})$. To prove Eq. \eqref{eqdefell}, it remains to show that the canonical linear injective map $ \frac{ \hz\cap\ker L }{\cz \cap\ker L} \to \frac{ \hz\cap L^{-1}\cz }{\cz}$ is surjective. This is true using Lemma \ref{lemopdiff}: for any $f \in \hz$ such that $Lf\in\cz$, there exists $P \in \cz$ such that $f+P\in 
\hz\cap\ker L$.

\bigskip

We fix a sufficiently large positive integer $m$ 
 and let $\lun =\ell+m-1$ (see \S \ref{subseclienfns} for details). The only difference with \cite{fns} is that $m$ may have to be larger in the present paper; $\lun$ plays the same role in the present paper as $\ell_0$ in \cite{fns}. 
 
Let $r$, $S$ be integers such that $0\leq r \leq S$. As in the introduction we denote by $F(z) = \sum_{k=0}^{+\infty} A_k z^k$ the $G$-function we are interested in, with $A_k\in\K$ for any $k$. In \cite{fns} we have used in a central way the series
$$J_F(z):= n!^{S-r}\sum_{k=0}^\infty \frac{k(k-1)\cdots (k-rn+1)}{(k+1)^S(k+2)^S\cdots (k+n+1)^S} \,A_k \,z^{ k+n+1}
$$
 which is denoted by $z^{n+1}T_F(1/z)$ in \cite{fns}. For any $n\geq \lun$ we have constructed polynomials $P_{u,s,n}(X) $ and $\widetilde{P}_{u,n}(X) $ in $ \K[X]$ of respective 
degrees $\le n $ and $\le n+1 +S(\ell-1)$ such that 
for any $z$ inside the open disk of convergence of $F$, 
$$
J_F(z)=\sum_{u=1}^{\lun}\sum_{s=1}^S P_{u,s,n}(z) F_{u}^{[s]}(z)+ \sum_{u=0}^{\mu-1} \widetilde{P}_{u,n}(z) (\theta^uF)(z) 
$$
(see Lemma \ref{lem600} below; recall that $\theta =z\frac{\dd}{\dd z}$). These polynomials play a central role in the present paper. We remark that, as in \cite{fns}, these polynomials depend on the value of $m$ we have chosen; this value is fixed throughout the present paper.

One of our main steps is to prove that they make up a solution of the following Pad\'e approximation problem which involves
$$
J_f(z):=\sum_{u=1}^{\lun}\sum_{s=1}^S P_{u,s,n}(z) f_{u}^{[s]}(z)+ \sum_{u=0}^{\mu-1} \widetilde{P}_{u,n}(z) (\theta^u f)(z) 
$$
for all solutions $f$ of the differential equation $Ly=0$; all functions $ f_{u}^{[s]} $ and $\theta^u f$ involved in this formula are holomorphic on $\dpr$ (see \S \ref{subsecdeffjs} for the definition of $ f_{u}^{[s]} $ in this setting).

\begin{Th} \label{thpade}
Let $F$ be a $G$-function and $L$ be a $G$-operator such that $LF=0$ and \eqref{eqopdiff} holds. Let $r$, $S$ be integers such that $0 \leq r \leq S$; 
assume that $S$ is sufficiently large in terms of $L$.

Then there exists an integer $\kappa$ (depending only on $L$, $r$, $S$) such that for any sufficiently large integer $n$, the polynomials $P_{u,s,n}(z)$ (for $1\leq u \leq \ell_0$ and $1\leq s \leq S$) and $\Pti_{u,n}(z)$ (for $0\leq u \leq \mu-1$) have the following properties.
\begin{enumerate}
\item[$(i)$] For any $f\in\ker L$ we have as $z\to 0$:
\begin{eqnarray} 
J_f(z) &=& O(z^{(r+1)n+1}) \mbox{ if $f$ is holomorphic at 0,} \label{eqpadehol} \\
J_f(z) &=& O(z^{n-\kappa}) \mbox{ otherwise.} \label{eqpadenonhol} 
\end{eqnarray}
\item[$(ii)$] For any $\alpha\in\Sigma\setminus\{0\}$ and any $f\in\ker L$ we have
\begin{equation}\label{eqpadealpha}
 \omal(J_f)(z) = O((z-\alpha)^{(S-r)n-\kappa}), \quad z\to \alpha.
\end{equation}
Moreover the left hand side of Eq. \eqref{eqpadealpha} is identically zero if, and only if, $f$ is holomorphic at $\alpha$.
\item[$(iii)$] For any $u\in\{1,\ldots,m-1\}$ and any $s\in \{1,\ldots,S\}$ we have
\begin{equation}\label{eqpetitdeg}
P_{u,s,n}(z) = O(z^{n+1-u}), \quad z\to 0.
\end{equation}
\end{enumerate}
\end{Th}

Theorem \ref{thpade} really involves a Pad\'e approximation problem since the left hand side of Eq. \eqref{eqpadealpha} is 
$$ \omal(J_f)(z) = \sum_{u=1}^{\lun}\sum_{s=1}^S P_{u,s,n}(z) \omal ( f_{u}^{[s]})(z)+ \sum_{u=0}^{\mu-1} \Pti_{u,n}(z) \omal (\theta^u f)(z) .$$
 
 \begin{Remark} It could be interesting to generalize Theorem \ref{thintroun} by restricting to values $F_n^{[s]}(z_0)$ for which $s$ has a given parity. This would probably involve a Pad\'e approximation problem similar to the one of Theorem \ref{thpade}, but with also vanishing conditions at infinity (as in \cite{JMPA}).  
 \end{Remark}

\subsection{Pad\'e approximation with non-holomorphic conditions} \label{subsecnonholopade}

\newcommand{\Id}{{\rm Id}}
\newcommand{\polP}{\Pi}

An important feature of the Pad\'e approximation problem of Theorem \ref{thpade} is that it involves vanishing conditions of non-holomorphic functions. In this section we explain how to count the number of linear equations equivalent to such a condition. Let us start with a well-known Pad\'e approximation problem, namely Beukers' for $\zeta(3)$. It is the following: find polynomials $P_1,\ldots,P_4$ of degree at most $n$ such that 
$$\left\{
\begin{array}{l}
R_1(z) := P_1 (z) \Li_2(1/z) + P_2 (z) \Li_1(1/z) + P_3 (z) = O(z^{-n-1}), \quad z\to\infty \\
\\
R_2(z) :=2 P_1 (z) \Li_3(1/z) + P_2 (z) \Li_2(1/z) + P_4 (z) = O(z^{-n-1}), \quad z\to\infty \\
\\
R_3(z) :=P_1 (z) \log(z) - P_2 (z) = O(1-z), \quad z\to1.
\end{array}\right.$$
 In Beukers' paper \cite{BeukersAperyPade} the last condition is $P_2(1)=0$; this equivalent formulation appears in \cite{JMPA}. The functions $R_1$ and $R_2$ are holomorphic at $\infty$; $R_3$ is holomorphic at 1. The first and last conditions are ``nice'' because $\tra\, ( 
\Li_2(1/z), \Li_1(1/z), 1 , 0)$ and $\tra \, ( \log(z), -1, 0, 0)$ are solutions of a common differential system $Y'=AY$. However the second one does not fit into this context. This makes it impossible to apply Shidlovsky's lemma to this Pad\'e approximation problem. Using non-holomorphic vanishing conditions will help us overcome this problem (see \cite{Pise, Huttner, Sorokincyclic} for other Pad\'e approximation problems where the same situation appears).

Since the derivative of $ \Li_2(1/z) \log(z) + 2 \Li_3(1/z) $ is $\frac{-1}z ( \Li_1(1/z) \log(z) + \Li_2(1/z))$, following \cite[\S 4.1]{Huttner}, we replace the condition on $R_2$ with 
\begin{multline*}
 R_1(z) \log(z) + R_2(z) = P_1 (z) \Big( \Li_2(1/z) \log(z) + 2 \Li_3(1/z) \Big) \\ 
 + P_2 (z) \Big( \Li_1(1/z) \log(z) + \Li_2(1/z)\Big) + P_3 (z) \log(z) + P_4(z) 
 = O(z^{-n-1}\log (z)), \quad z\to\infty.
\end{multline*}
Then we have three conditions (two at $\infty$ and one at 1) related to three solutions of the same differential system: the new one is 
$$
\tra \Big( \Li_2(1/z) \log(z) + 2 \Li_3(1/z), \quad \Li_1(1/z) \log(z) + \Li_2(1/z), \quad \log(z), \quad 1\Big).$$
This solution at $\infty$ is not holomorphic at $\infty$. It could look complicated, at first glance, to see how many linear equations are equivalent to this system. Indeed the condition $ R_1(z) \log(z) + R_2(z) =O(z^{-n-1}\log (z))$ itself means that both $R_1(z)$ and $ R_2(z) $ are $O(z^{-n-1})$, $z\to\infty$: it should be counted as two conditions involving holomorphic functions. However the condition on $R_1(z)$ already appears separately as the first vanishing condition of our Pad\'e approximation problem, so it has been counted already.

\bigskip

The following lemma elaborates upon this idea; it enables one to count the number of linear equations equivalent to such vanishing conditions (at a finite point $\alpha$; it would not be difficult to adapt it at $\infty$ in case it would be needed for another Pad\'e approximation problem). We shall apply it in \S \ref{subseccompte} to prove that the Pad\'e approximation problem of Theorem \ref{thpade} has (up to an additive constant) as many equations as its number of unknowns. Of course this computation is not really used to apply Shidlovsky's lemma, but it is useful to ensure that a Pad\'e approximation problem is reasonable.

\begin{lem} \label{lemcptesolE}
Let $N \geq 0$, $A\in M_q(\C(X))$ and $\alpha\in\C$; if $\alpha$ is a singularity of the differential system $Y'=AY$, assume it is a regular one. Let $E$ be a $\C$-vector space of local solutions at $\alpha$ of this differential system, such that $\omal Y \in E$ for any $Y\in E$. Then the conditions
\begin{equation} \label{eqcptesolE}
P_1(z) y_1(z) + \ldots + P_q(z) y_q(z) = O((z-\alpha)^{N}),
\end{equation}
for $Y = \, \tra (y_1,\ldots,y_q)\in E$, amount to $N \dim E + O(1)$ linear equations where $O(1)$ is bounded in terms of $A$ only.
\end{lem}

\begin{Remark} If all $Y\in E$ have components holomorphic at $\alpha$, Lemma \ref{lemcptesolE} is trivial: one writes 
Eq. \eqref{eqcptesolE} for all elements $Y$ of a basis of $E$. The point of Lemma \ref{lemcptesolE} is that the conditions \eqref{eqcptesolE} can be translated into $\dim E$ vanishing conditions of holomorphic functions.
\end{Remark}

\bigskip

\begin{proof} Let $\lambda_1$, \ldots, $\lambda_p$ denote the pairwise distinct eigenvalues of $\omal$, seen as a linear map $E\to E$. For any $i\in\unp$, let $E_i = \ker ((\omal - \lambda_i \Id)^{\dim E})$; then we have $E = E_1\oplus \ldots\oplus E_p$. 

For each $i\in\unp$ there exist cyclic subspaces $E_{i,u}$ (for $1\leq u \leq U_i$) such that $E_i = E_{i,1}\oplus \ldots\oplus E_{i,U_i}$. The assertion that 
$E_{i,u}$ is cyclic means that there exists $Y^{\{i,u\}}\in E_{i,u}$ such that $E_{i,u}$ is the set of all $\polP (\omal)(Y^{\{i,u\}})$ with $\polP \in\C[X]$. Now let $V_{i,u}$ denote the minimal degree of a non-zero polynomial $\polP \in\C[X]$ such that $\polP (\omal)(Y^{\{i,u\}}) =0$. Then the $\omal^v( Y^{\{i,u\}}) $, for $1\leq i \leq p$, $1\leq u \leq U_i$ and $0 \leq v \leq V_{i,u}-1$, make up a basis of $E$. Moreover the same property holds if each $\omal^v( Y^{\{i,u\}}) $ is replaced with $\polP _{i,u,v} (\omal) ( Y^{\{i,u\}}) $ where $\polP _{i,u,v}\in\C[X]$ is an arbitrary polynomial of degree $v$. 

In concrete terms, for any $i\in\unp$ we denote by $e_i$ the exponent at $\alpha$ of the differential system $Y'=AY$ with the least real part among those such that $\exp(2i\pi e_i) = \lambda_i$. Then we have 
$ Y^{\{i,u\}} = \, \tra ( y_1 ^{\{i,u\}}, \ldots, y_q ^{\{i,u\}})$ with 
$$y_j ^{\{i,u\}} (z) = (z-\alpha)^{e_i} \sum_{t=0} ^{ V_{i,u}-1} \kappa_{j,t} ^{\{i,u\}}(z) \frac1{t!} \log(z-\alpha)^t \mbox{ for any } j\in\unq,$$
where $\kappa_{j,t} ^{\{i,u\}}(z) $ is holomorphic at $\alpha$ for any $j$, $t$, and $\kappa_{j, V_{i,u}-1} ^{\{i,u\}}(z) $ is not identically zero for at least one $j\in\unq$. Choosing the above-mentioned polynomials $\polP _{i,u,v}$ in a suitable way, we obtain a basis of $E$ consisting of vectors $\polP _{i,u,v} (\omal) ( Y^{\{i,u\}}) = \tra ( y_1 ^{\{i,u,v\}}, \ldots, y_q ^{\{i,u,v\}})$ with the following expression:
$$y_j ^{\{i,u,v\}} (z) = (z-\alpha)^{e_i} \sum_{t=v} ^{ V_{i,u}-1} \kappa_{j,t} ^{\{i,u\}}(z) \frac1{(t-v)!} \log(z-\alpha)^{t-v} \mbox{ for any } j\in\unq.$$
The point here is that $\kappa_{j,t} ^{\{i,u\}}(z) $ does not depend on $v$. This basis is of the same form as those produced by Frobenius' method (see for instance \cite[Eq. (4.4)]{ateo}); the difference here is that $E$ is not (in general) the space of all solutions of the differential system.

This basis allows us now to count the number of linear equations imposed by Eq. \eqref{eqcptesolE} for any $Y \in E$, or equivalently for all $\polP _{i,u,v} (\omal) ( Y^{\{i,u\}}) $ as $i$, $u$, $v$ vary. Indeed let us fix $1\leq i \leq p$ and $1\leq u \leq U_i$. Then for any 
 $0 \leq v \leq V_{i,u}-1$ we have 
 $$\sum_{j=1}^q P_j(z) y_j ^{\{i,u,v\}} (z) = (z-\alpha)^{e_i} \sum_{t=v} ^{ V_{i,u}-1} \Big( \sum_{j=1}^q P_j(z) \kappa_{j,t} ^{\{i,u\}}(z) \Big) \frac1{(t-v)!} \log(z-\alpha)^{t-v} .$$
Therefore the vanishing conditions \eqref{eqcptesolE} for $\polP _{i,u,v} (\omal) ( Y^{\{i,u\}}) $ with $v=0, \ldots, V_{i,u}-1$ are equivalent (up to powers of $\log(z-\alpha)$) to the following $V_{i,u}$ vanishing conditions concerning functions {\em holomorphic at $\alpha$}: 
\begin{equation} \label{equnt}
 \sum_{j=1}^q P_j(z) \kappa_{j,t} ^{\{i,u\}}(z) = O((z-\alpha)^{N-e_i}), \quad t=0, \ldots, V_{i,u}-1.
 \end{equation} 
As long as $i$ and $u$ are fixed, Eq. \eqref{equnt} is equivalent to $N V_{i,u} + O(1)$ linear equations, because $e_1$, \ldots, $e_p$ depend only on $A$. Letting $i$ and $u$ vary, we obtain $N \dim E + O(1)$ linear equations: this concludes the proof of Lemma \ref{lemcptesolE}.
\end{proof}

\subsection{Counting equations and unknowns in Theorem \ref{thpade}} \label{subseccompte}

In this section we shall prove that in the Pad\'e approximation problem of Theorem \ref{thpade}, the difference between the number of unknowns and the number of equations is bounded from above independently from $n$. With this aim in view, we denote by $O(1)$ any quantity that can be bounded in terms of $L$, $r$, $S$ only (independently from $n$). A parallel but more precise argument will be used in \S \ref{subsecshidfin} to check the assumptions of Shidlovsky's lemma.

\bigskip

The unknowns are the coefficients of the following polynomials:
\begin{itemize}
\item $P_{u,s,n}$ of degree at most $n$, for $1\leq u \leq \lun$ and $1\leq s \leq S$, that is $(n+1)\lun S$ coefficients.
\item $\Pti_{u,n}$ of degree at most $n+1+S(\ell-1)$, for $0\leq u \leq \mu-1$, that is $\mu(n+2+ S(\ell-1))$ coefficients.
\end{itemize}
To sum up, this Pad\'e approximation problem involves
$$\Big( \mu + \lun S \Big) n + O(1)$$
coefficients. Let us count now the number of equations.

\bigskip

\begin{itemize}
\item Polynomial solutions of the differential equation $Ly=0$ have to be considered separately. Indeed for any non-zero $f\in\cz\cap\ker L$, Eq. \eqref{eqpadehol} seems to be equivalent to $(r+1)n+1$ linear equations. However this is not the case: as Eq. \eqref{eqJholo}
 shows, $J_f(z)$ is a polynomial of degree at most $n+O(1)$: Eq. \eqref{eqpadehol} means that $J_f(z)$ is identically zero. This should be understood as a system of $n+O(1)$ linear equations for each $f$ in a basis of $\cz\cap\ker L$: up to $O(1)$, this is the same number of equations as Eq. \eqref{eqpadenonhol}. 
 We would like to point out that, as far as we know, Theorem \ref{thpade} is the first Pad\'e approximation problem in this setting with this kind of conditions. The only new feature of the general version of Shidlovsky's lemma proved in \S \ref{secshid} below (with respect to that of \cite{SFcaract}) is to deal with this situation through the parameter $\roro$. 

\item The condition \eqref{eqpadenonhol}, taken for $f$ in a basis of $\ker L$ (including possible polynomial solutions), amounts to 
$$\mu n + O(1)$$
linear equations since $\mu=\dim\ker L$. This is not obvious a priori because elements of $\ker L$ are not always holomorphic at 0. Let us deduce this property from Lemma~\ref{lemcptesolE}. With this aim in mind, we 
 denote by $E$ the set of tuples
$$ \Big( (f^{[s]}_{u})_{m\leq u \leq \lun, 1\leq s \leq S}, \, ( \theta^u f )_{0\leq u \leq \mu-1}\Big)$$
with $f\in\ker L$. Since $\omz$ commutes with differentiation we have $\omz ( \theta^u f) = \theta^u (\omz f)$, and Lemma \ref{lempoidscommute} yields $\omz ( f^{[s]}_{u}) = (\omz f)^{[s]}_{u}$ for any $u\geq m$. Therefore $E$ is stable under $\omz$ because $\ker L$ is: Lemma \ref{lemcptesolE} with $N = n-\kappa$
 concludes the proof that Eq. \eqref{eqpadenonhol} amounts to 
 $\mu n + O(1)$ linear equations.

\item For each $f$ in a basis of $\hz\cap\ker L$, Eq. \eqref{eqpadehol} seems to be equivalent to $(r+1)n+O(1)$ linear equations. However Eq. \eqref{eqpadenonhol} has been taken into account already, so only $ r n+O(1)$ equations are new. Moreover, if $f$ is a polynomial this is misleading (see above): in this case Eq. \eqref{eqpadehol} amounts only to $O(1)$ new equations with respect to Eq. \eqref{eqpadenonhol}. To conclude, Eq. \eqref{eqpadehol} should be seen as a system of 
$$ rn \dim \Big( \frac{ \hz\cap\ker L }{\cz \cap\ker L}\Big) + O(1) = \ell r n +O(1)$$
new linear equations.

\item Let us consider Eq. \eqref{eqpadealpha} now. 
For $\alpha\in\Sigma\setminus\{0\}$ we denote by $E_\alpha$ the set of tuples
$$\Psi_\alpha ( f) = \Big( ( \omal ( f^{[s]}_{u}))_{1\leq u \leq \lun, 1\leq s \leq S}, \, ( \omal( \theta^u f ))_{0\leq u \leq \mu-1}\Big)$$
with $f\in \ker L$. Let $\hal$ denote the space of functions holomorphic at $\alpha$. If $f\in\hal\cap\ker L$ then all $f^{[s]}_{u}$ and all $ \theta^u f$ are holomorphic at $\alpha$, so that $f\in\ker\Psi_\alpha$. Now if $f\in\ker\Psi_\alpha\subset \ker L $ is non-zero then $f$ has (at most) a regular singularity at $\alpha$, and $\omal f = 0$ so that $f$ is meromorphic at $\alpha$: there exists $k\in\Z$ and $c\in\C\setminus\{0\}$ such that $f(z)\sim c (z-\alpha)^k$ as $z\to\alpha$. If $k<0$ then $f_u^{[-k]}$ has a logarithmic divergence at $\alpha$ for any $u$, which is impossible because $f\in\ker\Psi_\alpha$ and $-k \leq S$ (since the order of $\alpha$ as a pole of a solution of $L$ is bounded in terms of $L$). Therefore $k\geq 0$: this concludes the proof that $\ker\Psi_\alpha = \hal\cap\ker L$. Accordingly, the space $E_\alpha=\Im \Psi_\alpha$ is isomorphic to $\frac{\ker L}{ \hal\cap \ker L}$. 

 We have $\omal ( \omal( \theta^u f )) = \omal( \theta^u(\omal f ))$ since $\omal$ commutes with differentiation, and $\omal ( \omal ( f^{[s]}_{u})) = \omal ( (\omal f)^{[s]}_{u})$ using Lemma \ref{lempoidsalpha}. Therefore
 $E_\alpha$ is stable under $\omal$ because $ \ker L$ is. Lemma \ref{lemcptesolE} with $N = (S-r)n-\kappa$ shows that Eq. \eqref{eqpadealpha} amounts to 
 $$ (S-r)n \dim E_\alpha + O(1)$$ linear conditions. Now recall that 
 $L^{-1}( \hal )/ \hal $ is the vector space of microsolutions at $\alpha$; Kashiwara's theorem \cite{Kashiwara} (see also \S 1.2 of \cite{Pham}) asserts that this vector space has dimension $m_\alpha$, the multiplicity of $\alpha$ as a singularity of $L$. For any $f\in L^{-1}( \hal )$, Theorem 1 of \cite{fuchsien} provides $f_0 \in \hal$ such that $L(f+f_0)\in\C[z]$; then \eqref{eqopdiff} gives $P\in\C[z]\subset \hal $ such that $f+f_0+P\in\ker L$. This proves that the canonical injective map $\frac{ \ker L }{\hal \cap \ker L } \to 
\frac{ L^{-1}( \hal )}{\hal }$ is bijective, so that 
$$\dim E_\alpha = \dim \Big( \frac{ \ker L }{\hal\cap \ker L } \Big) = \dim \Big( \frac{ L^{-1}( \hal )}{\hal } \Big) = m_\alpha.$$
Denote by $\delta$ the degree of $L$, and by $\omega$ the multiplicity of $0$ as a singularity of $L$. Since $0$ is a regular singularity we have $\delta = \omega + \sum_{\alpha\neq 0} m_\alpha$ so that $ \sum_{\alpha\neq 0} m_\alpha = \delta - \omega =\ell $ (using the definition of $\ell$ given in \cite{fns}, which is equivalent to the one used in the present paper: see the remark after Eq. \eqref{eqdefell}). Therefore combining Eqns. \eqref{eqpadealpha} as $\alpha$ varies in $\Sigma\setminus\{0\}$ amounts to $\ell (S-r)n +O(1)$ linear equations.
 \item
Eq. \eqref{eqpetitdeg} amounts to $n+1-u$ linear conditions on $P_{u,s,n}$; as $u$ and $s$ vary, they provide $(m-1)Sn +O(1)$ equations.
\end{itemize}
To sum up, Theorem \ref{thpade} amounts to $(\mu +\ell S + (m-1)S) n + O(1)$ linear equations: up to an additive constant (independent of $n$), this is exactly the number of unknowns since $\ell +m-1 = \lun$.

\section{Proof of Theorem \ref{thpade}} \label{secpadepreuve}

In this section we prove Theorem \ref{thpade} stated in \S \ref{subsecnotationspade}. Our approach is based on the main technical result of \cite{fns}, recalled (together with the notation) in \S \ref{subseclienfns}. Following \cite[Lemma 4]{fns} we deal in \S \ref{subsecholenzero} with holomorphic solutions at 0. Then we move to the other local vanishing conditions of Theorem \ref{thpade}, involving other solutions at 0 (\S \ref {subsecnonholoother}) and solutions at non-zero finite singularities (\S \ref{subsecannulationsingularites}). 

\bigskip

Throughout this section we keep the notation of \S \ref{subsecnotationspade}; in particular we use the $G$-operator $L$ provided by Lemma \ref{lemopdiff}.

\subsection{Prerequisites and notations from \cite{fns}} \label{subseclienfns}

We recall that $\ell$ has been defined in \S \ref{subsecnotationspade} (see Eq. \eqref{eqdefell}); we also define 
\begin{equation} \label{eqdefnz}
\capazero := \min(e_1,\ldots,e_\mu,0)
\end{equation}
where $e_1,\ldots,e_\mu$ are the exponents of $L$ at zero (including possibly non-integer ones). Denoting by 
$\widehat{f}_1$, \ldots, $\widehat{f}_\eta$ the {\em integer} exponents at $\infty$ (with $\eta=0$ if there isn't any),
throughout this paper we fix an integer $m\geq 1$ such that 
\begin{equation} \label{eqhypmgros}
m > -\capazero \mbox{ and } m > \widehat{f}_{j} - \ell \; \mbox{ for all } \; 1\le j \leq \eta
\end{equation}
and let $\lun = \ell+m-1$. The only difference with \cite{fns} is that our assumption \eqref{eqhypmgros} on $m$ is more restrictive than the one of \cite{fns} (where only the second inequality appears), so that we may have to take a larger value of $m$. The integer $\lun$ plays the same role as $\ell_0$ in \cite{fns}.   The reason for this is that in \cite{fns} only a specific solution $F$, holomorphic at 0, is involved. On the opposite, in the present paper we have to consider all solutions, holomorphic at 0 or not, of the differential equation $Ly=0$. However this larger value of $m$ does not have any impact on Theorem \ref{thintroun}. 

\bigskip

By definition of $\capazero$, any local solution of $Lf=0$ at the origin can be written as 
$$f(z) = \sum_{k\in \Q \atop k\geq \capazero} \sum_{i=0}^{e} a_{k,i} z^k \log (z)^i$$
since $L$ is Fuchsian with rational exponents. Then for any $n \geq m > -\capazero$ and any $s\geq 1$, all integrals involved in the definition of 
 $f_{n}^{[s]}$ (see \S \ref{subsecdeffjs}) are convergent integrals. Therefore the following proposition can be proved exactly as in \cite{fns} (where in $(i)$ $f$ is assumed to be holomorphic at 0); recall that $\dpr$ has been defined after Lemma \ref{lemopdiff} in \S \ref{subsecnotationspade}.

\begin{Prop} \label{prop:1} For any $s \geq 1$ and any $n\geq \nun$:

$(i)$ There exist some algebraic numbers $\ppP_{j,t,s,n}\in \K$, and some polynomials $\qqQ_{j,s,n}(z) $ in $ \K[z]$ of degree at most $n+s(\ell -1)$, which depend also on $L$ (but not on $f$), such that for any $f\in\ker L$ and any $z\in\dpr$:
\begin{equation}\label{eq:406}
f_n^{[s]}(z)=\sum_{t=1}^{s}\sum_{j=\nun}^{\lun} \ppP_{j,t,s,n}f_j^{[t]}(z) + \sum_{j=0}^{\mu-1} \qqQ_{j,s,n}(z)(\theta^j f)(z) .
\end{equation}

$(ii)$ All Galois conjugates of all the numbers $ \ppP_{j,t,s,n} $ ($m \leq j\le\lun$, $t\le s$), and all Galois conjugates of all the coefficients of the polynomials $\qqQ_{j,s,n}(z)$ ($j\le \mu-1$), have modulus less than $H(s,n)>0$ with 
$$
\limsup_{n\to + \infty } H(s,n)^{1/n} \le C_1^s
$$ 
for some constant $C_1\ge 1$ which depends only on $L$.

$(iii)$ Let $D(s,n)>0$ denote the least common denominator of the algebraic numbers $\ppP_{j,t,s,n'}$ ($m \leq j\le\lun$, $t\le s$, $n'\leq n$) and of the coefficients of the polynomials $\qqQ_{j,s,n'}(z)$ ($j\le \mu-1$, $n'\le n$); then
$$
\limsup_{n\to + \infty } D(s,n)^{1/n} \le C_2^s
$$ 
for some constant $C_2\ge 1$ which depends only on $L$.
\end{Prop}

\begin{Remark} \label{remcoro} The difference between this result and \cite[Proposition 1]{fns} is that it applies to any solution of the differential equation $Ly=0$, whereas in \cite[Proposition 1]{fns} $f$ is assumed to be holomorphic at 0. This is the reason why $m$ has to be chosen larger in Eq. \eqref{eqhypmgros}. To deduce Corollary \ref{coro:tanintro1} from Theorem \ref{thintroun} in the introduction, we have used the analytic continuation to $\mathcal{D}_F$ of Identity (5.2) in \cite{fns} because with the notation of \cite{fns} it is possible in this example to choose $m=\ell_0=1$, whereas for some values of the parameters Eq. \eqref{eqhypmgros} does not hold with $m=1$ (so that we cannot choose $\lun=1$). This is not a problem since the deduction of Corollary \ref{coro:tanintro1} from Theorem \ref{thintroun} involves only $F$ and no other solution of the differential equation.
\end{Remark}

\subsection{Holomorphic solutions at 0} \label{subsecholenzero}

In this section we prove Eqns. \eqref{eqpadehol} and \eqref{eqpetitdeg} of Theorem \ref{thpade} involving holomorphic solutions at 0. They follow from Lemma \ref{lem600} which is essentially the construction of $P_{u,s,n}(X) $ and $\widetilde{P}_{u,n}(X) $ in \cite{fns}.

\bigskip

Let $f$ be a function, holomorphic on $\dpr$ and at 0, such that $Lf=0$. Denoting by $ \sum_{k=0}^\infty a_k z^k$ its local expansion at 0, we recall from \S \ref{subsecnotationspade} that 
\begin{equation}\label{eqJholo}
J_f(z) = n!^{S-r}\sum_{k=0}^\infty \frac{k(k-1)\cdots (k-rn+1)}{(k+1)^S(k+2)^S\cdots (k+n+1)^S} \,a_k \,z^{k+n+1} .
\end{equation}

The following lemma is essentially \cite[Lemma 4]{fns}. We copy the proof because it provides an explicit construction of the polynomials $P_{u,s,n}(X) $ and $\widetilde{P}_{u,n}(X) $ (see Eqns. \eqref{eq:101} and \eqref{eq:102}) that will be used several times later.

\begin{lem} \label{lem600} Let $n\ge \lun$. There exist some polynomials $P_{u,s,n}(X) $ and $\widetilde{P}_{u,n}(X) $ in $ \K[X]$ of respective 
degrees $\le n $ and $\le n+1 +S(\ell-1)$ such that, 
for any $f\in\ker L$ holomorphic at 0 and any $z \in\dpr$:
$$
J_f(z)=\sum_{u=1}^{\lun}\sum_{s=1}^S P_{u,s,n}(z) f_{u}^{[s]}(z)+ \sum_{u=0}^{\mu-1} \widetilde{P}_{u,n}(z) (\theta^uf)(z) . 
$$
Moreover:
\begin{itemize}
\item The polynomials $ P_{u,s,n}(z) $ and $\widetilde{P}_{u,n}(z) $ depend only on $L$ but not on $f$.
\item For $u \leq m-1 $ and any $s$, we have $P_{u,s,n}(z) = c_{u,s,n} z^{n+1-u}$ with $ c_{u,s,n} \in\Q$.
\end{itemize}
\end{lem}

\bigskip

This lemma implies directly Eqns. \eqref{eqpadehol} and \eqref{eqpetitdeg} of Theorem \ref{thpade}.

\bigskip

\begin{Remark} With the notation of \cite{fns} we have
$$P_{u,s,n}(z) = z^{n+1} C_{u,s,n}(1/z)\hspace{1.5cm} \mbox{ and } \hspace{1.5cm} \Pti_{u,n}(z) = z^{n+1+S(\ell-1)} \widetilde C_{u,n}(1/z) . $$
\end{Remark}

\bigskip

\begin{proof}
We have the partial fractions expansion in $k$: 
\begin{equation}\label{eq:100}
n!^{S-r}\frac{k(k-1)\cdots (k-rn+1)}{(k+1)^S(k+2)^S\cdots (k+n+1)^S} = \sum_{j=1}^{n+1} \sum_{s=1}^S \frac{c_{j,s,n}}{(k+j)^s}
\end{equation}
for some $c_{j,s,n}\in \mathbb Q$, which also depend on $r$ and $S$. By analytic continuation it is enough to prove Lemma \ref{lem600} when $|z| < R$, where $R$ is the radius of convergence of the local expansion $ \sum_{k=0}^\infty a_k z^k$ of $f(z)$ around 0. Recall that $
f_j^{[s]}(z) = \sum_{k\in \Q , \, k\geq \capaf} \frac{a_k}{(k+j)^s}z^{k+j}
$, so that 
$$
J_f(z)=\sum_{j=1}^{n+1} \sum_{s=1}^S c_{j,s,n} z^{n+1-j}f_{j}^{[s]}(z).
$$
Since $n\ge \lun\ge m$, by Proposition~\ref{prop:1} we have
\begin{align*}
J_f(z)&=\sum_{j=1}^{\lun}\sum_{s=1}^S c_{j,s,n} z^{n+1-j}f_{j}^{[s]}(z) + \sum_{j=\lun+1}^{n+1} \sum_{s=1}^S c_{j,s,n} z^{n+1-j}f_{j}^{[s]}(z)
\\
&=\sum_{j=1}^{\lun}\sum_{s=1}^S c_{j,s,n} z^{n+1-j}f_{j}^{[s]}(z) \\
&\qquad + \sum_{j=\lun+1}^{n+1} \sum_{s=1}^S c_{j,s,n}z^{n+1-j} \left(\sum_{t=1}^{s}\sum_{u=\nun}^{\lun} \ppP_{u,t,s,j}f_u^{[t]}(z) 
+ \sum_{u=0}^{\mu-1} \qqQ_{u,s,j}(z)(\theta^u f)(z) \right)
\\
&=\sum_{u=1}^{\lun}\sum_{s=1}^S P_{u,s,n}(z) f_{u}^{[s]}(z)+ \sum_{u=0}^{\mu-1} \widetilde{P}_{u,n}(z) (\theta^uf)(z) 
\end{align*}
where 
\begin{eqnarray}
P_{u,s,n}(z) &:=& c_{u,s,n}z^{n+1-u}+\sum_{j=\lun+1}^{n+1}\sum_{\sigma=s}^S z^{n+1-j} c_{j,\sigma,n}\ppP_{u,s,\sigma,j}, \label{eq:101} \\
\widetilde{P}_{u,n}(z) &:=&\sum_{j=\lun+1}^{n+1} \sum_{s=1}^S c_{j,s,n} z^{n+1-j}\qqQ_{u,s,j}(z). \label{eq:102}
\end{eqnarray}
In Eq. \eqref{eq:101} we agree that $\ppP_{u,s,\sigma,j} = 0$ if $1\leq u \leq \nun-1$, so that $P_{u,s,n}(z) = c_{u,s,n} z ^{n+1-u}$.

The assertion on the degree of these polynomials is clear from their expressions since $\qqQ_{u,s,j}(z)$ is a polynomial of degree at most $j+s(\ell-1)$.
\end{proof}

\subsection{Other local solutions at 0} \label{subsecnonholoother}

In this section we prove condition \eqref{eqpadenonhol} in Theorem \ref{thpade} $(i)$ involving (non-holomorphic) local solutions at 0, as a direct consequence of Lemma \ref{lemvraiesol} below. 

\bigskip

Since $L$ is Fuchsian with rational exponents, any solution of $Lf=0$ can be written around $z=0$ as 
\begin{equation}\label{eqNil2}
f(z) = \sum_{k\in \Q \atop k\geq \capaf} \sum_{i=0}^{e} a_{k,i} z^k \log (z)^i.
\end{equation}
Moreover, by definition of $\capazero$ (see \S \ref{subseclienfns}), we may choose $\capaf = \capazero$.
Using the polynomials $P_{u,s,n}(X) $ and $\widetilde{P}_{u,n}(X) $ provided by Lemma \ref{lem600} we may consider (as in \S \ref{subsecnotationspade})
\begin{equation}\label{eq:Jf3010}
J_f (z)=\sum_{u=1}^{\lun}\sum_{s=1}^S P_{u,s,n}(z) f_{u}^{[s]}(z)+ \sum_{u=0}^{\mu-1} \widetilde{P}_{u,n}(z) (\theta^uf)(z).
\end{equation}

\begin{lem} \label{lemvraiesol} If $n\ge \lun$ then for any $f\in\ker L$ we have, as $z\to 0$:
$$
J_f(z)=O(z^{ n+1+\capazero} \log (z)^{e+S}). 
$$
\end{lem}

In particular we deduce that $J_f(z)=O(z^{ n+\capazero})$ as $z\to0$; therefore Lemma \ref{lemvraiesol} implies condition \eqref{eqpadenonhol} in Theorem \ref{thpade} $(i)$.

\begin{proof} Let us consider 
$$
\widetilde {J_f} (z):=\sum_{j=1}^{n+1} \sum_{s=1}^S c_{j,s,n} z^{n+1-j}f_{j}^{[s]}( z).
$$
Since $n\geq \lun$, we have 
\begin{equation*}
\widetilde {J_f} (z) =\sum_{j=1}^{\lun}\sum_{s=1}^S c_{j,s,n} z^{n+1-j}f_{j}^{[s]}( z) + \sum_{j=\lun+1}^{n+1} \sum_{s=1}^S c_{j,s,n} z^{n+1-j}f_{j}^{[s]}(z).
\end{equation*}
The same computation as in the proof of Lemma \ref{lem600}, based on Proposition~\ref{prop:1}, yields 
\begin{align*}
\widetilde {J_f} (z) &=\sum_{j=1}^{\lun}\sum_{s=1}^S c_{j,s,n} z^{n+1-j}f_{j}^{[s]}( z) \\
& + \sum_{j=\lun+1}^{n+1} \sum_{s=1}^S c_{j,s,n}z^{n+1-j} \left(\sum_{t=1}^{s}\sum_{u=\nun}^{\lun} \ppP_{u,t,s,j}f_u^{[t]}(z) 
 + \sum_{u=0}^{\mu-1} \qqQ_{u,s,j}(z)(\theta^u f)(z) \right)
\\
&=\sum_{u=1}^{\lun}\sum_{s=1}^S P_{u,s,n}(z) f_{u}^{[s]}(z)+ \sum_{u=0}^{\mu-1} \widetilde{P}_{u,n}(z) (\theta^u f)(z) 
\end{align*}
using Eqns. \eqref{eq:101} and \eqref{eq:102}, so that $\widetilde {J_f} (z) = J_f(z)$. Now we have $f(z) = O(z^{\capazero} \log (z)^e)$ as $z\to0$ since Eq. \eqref{eqNil2} holds with $\capaf = \capazero$, so that $f_j^{[s]}(z) = O(z^{\capazero +j} \log (z)^{e+s})$ for any $j$ and any $s$. Therefore 
 $ \widetilde {J_f} (z) =O(z^{ \capazero+ n+1}\log (z)^{e+S} )$, and the same property holds for $ J_f(z)$. 
\end{proof}

\bigskip

\begin{Remark} \label{remsanslog}
Assume that the local expansion of $f\in\ker L$ around the origin is given by $f(z) = \sum_{k\in\Q , k \geq \capaf }^\infty a_{k } z^k$ with $a_k=0$ for any $k\in\Z$ such that $k < 0$. Then as in the holomorphic case (see \S \ref{subsecholenzero}) we obtain from \eqref{eq:Jf3010}, by combining the proofs of Lemmas \ref{lem600} and \ref{lemvraiesol}:
\begin{equation*}
J_{f}(z) = n!^{S-r} \sum_{k\in\Q \atop k \geq \capaf }^\infty \frac{k(k-1)\cdots (k-rn+1)}{(k+1)^S(k+2)^S\cdots (k+n+1)^S} a_k z^{ k+n+1} \end{equation*}
where we omit the terms corresponding to $k\in\Z$, $k\leq 0$. We refer to Lemma \ref{lemexplicite} below for the general expression of $J_f(z)$ in terms of the local expansion of $f$ around 0.
\end{Remark}

\subsection{Vanishing conditions at singularities} \label{subsecannulationsingularites}

In this section we conclude the proof of Theorem \ref{thpade} by proving assertion $(ii)$. We fix $\alpha\in
\Sigma \setminus\{0\}$ and assume that $r<S$ (otherwise it holds trivially). Given $f\in\ker L$, Theorem 1 of \cite{fuchsien} provides a function $h$, holomorphic at 0 and at all singularities $\beta\in \Sigma\setminus\{\alpha\}$, such that $h-f$ is holomorphic at $\alpha$ and $Lh\in\C[z]$.
Using \eqref{eqopdiff} to add a suitable polynomial (if necessary), we may assume that $Lh=0$. 
 If $h$ is holomorphic at $\alpha$ then so is $f$, and the left hand side of Eq. \eqref{eqpadealpha} is identically zero; from now on we assume that $h$ is not holomorphic at $\alpha$. Since $h$ is holomorphic at 0, Lemma \ref{lem600} yields 
\begin{equation}\label{tfatourner}
J_h(z) = \sum_{u=1}^{\lun}\sum_{s=1}^S P_{u,s,n}(z) h_{u}^{[s]}(z)+ \sum_{u=0}^{\mu-1} \Pti_{u,n}(z) (\theta^u h)(z) 
\end{equation}
for any $z\in\dpr$, with
\begin{equation}\label{tfaderiver}
J_h(z) = n!^{S-r}\sum_{k=0}^\infty \frac{(k-rn+1)_{rn}}{(k+1)_{n+1}^S}\, h_k \,z^{k+n+1}
\end{equation}
for $|z| < |\alpha|$, where $h(z)=\sum_{k=0}^\infty h_k z^k$; recall that 
 $\alpha$ is the only finite singularity of $h$. Transference theorems (see for instance \cite[\S VI.5]{Bible} or \cite[\S 6.2]{gvalues}) provide $t\in\Q$, $e \in \N$, $J \geq 1$, $d_1,\ldots,d_J \in \C\etoile$ and pairwise distinct $\zeta_1,\ldots,\zeta_J \in \C $ with $| \zeta_1 | = \ldots = | \zeta_J | = 1$ such that 
\begin{equation}\label{eqaksing}
h_k = |\alpha| ^{-k} k^t (\log k)^{e} \Big( \chi_k + o(1)\Big) \mbox{ as $k \to\infty$, with } \chi_k = \sum_{j=1}^J d_j \zeta_j^k.
\end{equation}
 
Assume that $n$ is sufficiently large. Since $r<S$, Eq. \eqref{eqaksing} shows that the series \eqref{tfaderiver} converges absolutely for any $z$ such that $|z| \leq |\alpha| $, so that $J_h(z) $ has no pole of modulus $ |\alpha| $. However Eq. \eqref{eqaksing} proves that the $k$-th Taylor coefficient of $J_h(z) $ at the origin grows essentially like $ |\alpha| ^{-k}$ (see \cite[Lemma 6]{gvalues}). Therefore $J_h(z) $ has a (non-polar) singularity of modulus $ |\alpha|$; Eq. \eqref{tfatourner} shows that it is also a singularity of $h$, so that this non-polar singularity of $J_h(z) $ is $\alpha $.

To conclude the proof, we notice that the integer $\kappa$ in Theorem \ref{thpade} may be increased to ensure that $\kappa > t -S +1$ (since the integer $t$ in Eq. \eqref{eqaksing} can be bounded in terms of $L$ only). Let $p$ be a non-negative integer less than or equal to $(S-r)n - \kappa$, so that $p+ t + rn - S(n+1) < -1$. Then Eq. \eqref{eqaksing} shows that the $p$-th derivative of $J_h(z) $ at $z = \alpha$ is defined by an absolutely convergent series (obtained by differentiating Eq. \eqref{tfaderiver}), so that it has a finite limit as $z\to \alpha$. Applying Lemma \ref{lempoids} yields 
$$\omal (J_f)(z) = o (( z-\alpha)^{ (S-r)n - \kappa }) \mbox{ as } z\to\alpha.$$
Using Eq. \eqref{tfatourner} we obtain
$$ \sum_{u=1}^{\lun}\sum_{s=1}^S P_{u,s,n}(z) \omal ( h_{u}^{[s]})(z)+ \sum_{u=0}^{\mu-1} \Pti_{u,n}(z) \omal (\theta^u h)(z) = o (( z-\alpha)^{ (S-r)n - \kappa }) \mbox{ as } z\to\alpha.$$
Now $h-f$ is holomorphic at $\alpha$: so are $h_{u}^{[s]} - f_{u}^{[s]}$ and $\theta^u h - \theta^u f$, and therefore $h$ may be replaced with $f$ in this equation.
Replacing also trivially the symbol $o$ with $O$, this yields Eq. \eqref{eqpadealpha} for $f$. 
Moreover the function in the left hand side is $\omal (J_h)(z) $: it is not identically zero since $ J_h(z) $ has a non-polar singularity at $\alpha$. This concludes the proof of Theorem \ref{thpade}.

\section{A general version of Shidlovsky's lemma} \label{secshid}

In this section we state and prove a general version of Shidlovsky's lemma. The point is to adapt the one of \cite{SFcaract} (based upon the approach of Bertrand-Beukers \cite{BB} and Bertrand \cite{DBShid}) so as to take into account polynomial solutions with zero Pad\'e remainders. We consider a general setting (see \S \ref{subsecsetting}) because we hope this result can be useful in other situations.

\subsection{Setting} \label{subsecsetting}

Given $\sigma\in \C\cup\{\infty\}$, recall that the Nilsson class at $\sigma$ is the set of finite sums
$$
f(z) = \sum_{e\in\C} \sum_{i\in\N} \lambda_{i,e} \, h_{i,e}(z) (z-\sigma)^e \log(z-\sigma)^i
$$
where $\lambda_{i,e}\in\C$, $h_{i,e}$ is holomorphic at $\sigma$, and $z-\sigma$ should be understood as $1/z$ if $\sigma=\infty$. If such a function $f(z)$ is not identically zero, we may assume that $h_{i,e}(\sigma)\neq 0$ for any $i$, $e$; then the generalized order of $f$ at $\sigma$, denoted by $\ord_\sigma f$, is the minimal real part of an exponent $e$ such that $\lambda_{i,e}\neq 0$ for some $i$.

\bigskip

Let $q$ be a positive integer, and $A \in M_q(\C(z))$. We fix $P_1,\ldots,P_q\in\C[z]$ and $n \in\N=\{0,1,2,\ldots\}$ such that 
 $\deg P_i \leq n$ for any $i$. Then with any solution $Y = \tra (y_1,\ldots,y_q)$ of the differential system $Y'=AY$ is associated a remainder $R(Y)$ defined by 
$$
R(Y)(z) = \sum_{i=1}^q P_i(z) y_i(z).$$

\medskip

Let $\Sigma$ be a finite subset of $ \C\cup\{\infty\}$. This will be the set of points where vanishing conditions appear. We do not assume any relationship\footnote{To help the reader compare with \cite{SFcaract}, the notation of this section is independent from the one in the previous sections. In particular $\Sigma$, $\alpha$, $n$ and $J_\sigma$ have different meanings (see \S \ref{subsecshidfin}).} between $\Sigma$ and the set of singularities of the differential system $Y' =AY$ (even though interesting Pad\'e approximation problems often involve vanishing conditions at singularities, as in Theorem \ref{thpade}). 
For each $\sigma\in\Sigma$, let 
$(Y_j)_{j\in J_\sigma}$ be a family of solutions of $Y'=AY$ such that the functions $R(Y_j)$, $j\in J_\sigma$, are $\C$-linearly independent and 
belong to the Nilsson class at $\sigma$.

\medskip

We agree that $J_\sigma=\emptyset$ if $\sigma\not\in \Sigma$, and we also consider a finite set $J$ and a family $(Y_j)_{j\in J}$ of 
 solutions of $Y'=AY$ such that $R(Y_j)=0$ for any $j\in J$; we assume these solutions $Y_j$ to be $\C$-linearly independent, and 
to belong to the Nilsson class at $\sigma$. We let $\roro = \Card J$; the case $\roro = 0$ is treated in \cite{SFcaract}. Such solutions $Y_j$ with zero Pad\'e remainders appear in Theorem~\ref{thpade} if $r\geq 1$ and the differential operator $L$ has non-zero polynomial solutions (see the proof of Proposition \ref{propshid} in \S \ref{secapplishid} below). They are the only reason why \cite[Theorem 2]{SFcaract} does not apply to our setting.

\bigskip

At last we 
 let $M(z) = (P_{k,i}(z))_{1\leq i,k \leq q} \in M_q(\C(z))$ where the rational functions $P_{k,i}\in\C(z)$ are defined for $k \geq 1$ and $1 \leq i \leq q$ by
\begin{equation} \label{eqdefpki}
\left( \begin{array}{c} P_{k,1} \\ \vdots \\ P_{k,q} \end{array}\right) = \left(\frac{\dd}{\dd z} + \, \tra A\right)^{k-1} 
 \left( \begin{array}{c} P_{1} \\ \vdots \\ P_{q} \end{array}\right).
 \end{equation}
  Obviously the poles of the coefficients $P_{k,i}$ of $M$ are among those of $A$. These rational functions $P_{k,i}$ play an important role because they are used to differentiate the remainders \cite[Chapter 3, \S 4]{Shidlovski}:
\begin{equation} \label{eqderiR}
R(Y)^{(k-1)}(z) = \sum_{i=1}^q P_{k,i}(z) y_i(z). 
\end{equation}

\subsection{Statement of Shidlovsky's lemma}

In the setting of \S \ref{subsecsetting}, let $\tau\in\Z$ be such that 
\begin{equation} \label{eqhypdetnn}
\sum_{\sigma\in\Sigma}\sum_{j\in J_{\sigma}} \ord_\sigma(R(Y_j)) \geq (n+1) ( q - \Card J_\infty - \roro) -\tau .
\end{equation}
The first result we shall prove below is the existence of a positive constant $\cstun$, which depends only on $A$ and $\Sigma$,
 such that:
 \begin{itemize}
 \item We have $\tau \geq -\cstun$.
 \item If $ \tau \leq n - \cstun$ then $\rk (M(z)) = q-\roro$.
\end{itemize}
Here and below, we denote by $\rk$ the rank of a matrix.
This result generalizes the functional part of Shidlovsky's lemma, namely $\det M(z)$ is not identically zero (if $\roro=0$). Its proof relies on the functional Shidlovsky's lemma of \cite{SFcaract}, which is itself based upon the approach of Bertrand-Beukers \cite{BB} and Bertrand \cite{DBShid}.

\bigskip

The next step is to evaluate at a given point $\alpha$, going from functional to numerical linear forms (see \cite[Chapter 3, Lemma 10]{Shidlovski} for the classical setting). As in \cite{SFcaract} we allow $\alpha$ to be a singularity of the differential system $Y'=AY$, and/or an element of $\Sigma$ (eventhough in the proof of Theorem \ref{thintroun}, $\alpha\not\in\Sigma$ is not a singularity).

\begin{Th} \label{thshid}
There exists a positive constant $\cstnum$, which depends only on $A$ and $\Sigma$, with the following property. Assume that, for some $\alpha\in\C$:
 \begin{itemize}
\item[$(i)$] If $\alpha$ is a singularity of the differential system $Y'=AY$, it is a regular one and all non-zero exponents at $\alpha$ have positive real parts.
 \item[$(ii)$] Eq. \eqref{eqhypdetnn} holds for some $\tau $ with $0 \leq \tau \leq n - \cstun$.
\item[$(iii)$] All rational functions $P_{k,i}$, with $1\leq i \leq q$ and $1 \leq k < \tau + \cstnum$, are holomorphic at $z=\alpha$.
\end{itemize}
Then the matrix $(P_{k,i}(\alpha))_{1\leq i \leq q, 1\leq k < \tau + \cstnum} \in M_{q,\tau + \cstnum-1}(\C)$ has rank at least $q-\roro-\Card J_\alpha$.
\end{Th}

In particular, assertion $(i)$ holds if the differential system $Y'=AY$ 
 has a basis of local solutions at $\alpha$ with coordinates in \mbox{$\C[\log(z-\alpha)][[(z-\alpha)^e]]$} for some positive rational number $e$.
 
 The lower bound on $\rk \big(P_{k,i}(\alpha))\big)$ provided by Theorem \ref{thshid} is an equality in many special cases, including the Pad\'e approximation problem of Theorem \ref{thpade} (see Proposition \ref{propshid} below) and the one studied in \cite[\S 4]{SFcaract}; it would be interesting to know if this is always the case.

\subsection{Proof of Shidlovsky's lemma} \label{subsecpreuveshid}

In this section we prove Theorem \ref{thshid}. The proof falls into 3 steps.
 
\bigskip

\noindent {\bf Step 1: } $M(z) \in M_q(\C(z))$ has rank at least $q-\roro$.

\medskip

As in \cite{Shidlovski}, there is a non-trivial linear relation with coefficients in $\C(z)$ between the $\rk(M)+1$ first columns of $M$; this provides a differential operator $L$ of order $\mu = \rk(M)$ such that $L(R(Y))=0$ for any solution $Y$ of the differential system $Y'=AY$, because we have
 $$\tra \, Y M = \left( \begin{array}{cccc} R(Y) & R(Y)' & \ldots & R(Y)^{(q-1)}\end{array}\right).$$
In particular, $L(R(Y_j))= 0 $ for any $\sigma$ and any $j\in J_\sigma$. Therefore \cite[Theorem 3.1]{SFcaract} yields 
$$
\sum_{\sigma\in\Sigma}\sum_{j\in J_{\sigma}} \ord_\sigma(R(Y_j)) \leq (n+1) (\mu - \Card J_\infty) + \cstun
$$
where $\cstun$ is a constant that depends only on $A$ and $\Sigma$. Together with Eq. \eqref{eqhypdetnn} and the assumption $\tau\leq n -\cstun$, this inequality implies $\mu \geq q -\roro$. 

\bigskip

\noindent {\bf Step 2: } Determination of minors up to factors of bounded degree.

\medskip

From now on we denote by $\mu$ the rank of $M(z)$, and we consider a $\mu\times\mu$ submatrix $M_0(z)$ of $M(z)$ such that $\det M_0(z)\neq 0$. Step 1 yields $\mu\geq q-\roro$, and we shall prove that equality holds.

Let $S$ denote the set of finite singularities of the differential system $Y'=AY$, i.e. poles of coefficients of $A$. For any $s\in S$, let $N_s$ denote the maximal order of $s$ as a pole of a coefficient of $A$; let $N_s = 0$ for $s\in\C\setminus S$. Then Eq. \eqref{eqdefpki} shows that $(z-s)^{(k-1)N_s}P_{k,i}(z) $ is holomorphic at $z=s$ for any $k\geq 1$ and any $i\in\unq$. Therefore $\det M_0(z) \cdot \prod_{s\in S} (z-s)^{q(q-1)N_s}$ has no pole: is it a polynomial.

Now let $\sigma\in \Sigma$, and denote by $T_\sigma \in M_{\roro+\Card J_\sigma, q}(\hol)$ the matrix with rows $\tra \, Y_j$, $j \in J\sqcup J_\sigma$; here $J\sqcup J_\sigma$ is the disjoint union of $J$ and $J_\sigma$, in which any element of $ J\cap J_\sigma$ appears twice (but we shall prove shortly that there is no such element).
 Let us prove that these rows are linearly independent over $\C$, so that $\rk T_\sigma = \roro+\Card J_\sigma$. Assume that $\sum_{j \in J\sqcup J_\sigma} \lambda_j Y_j = 0$; then $\sum_{j \in J_\sigma} \lambda_j R(Y_j) = 0$, so that $ \lambda_j = 0$
for any $j \in J_\sigma$ because the remainders $R(Y_j)$, $j \in J_\sigma$, are $\C$-linearly independent. Therefore $\sum_{j \in J } \lambda_j Y_j = 0$, and finally $ \lambda_j = 0$
for any $j \in J $ because the solutions $Y_j$, $j \in J $, are $\C$-linearly independent. This concludes the proof that $\rk T_\sigma = \roro+\Card J_\sigma$; in the same time we have proved that $ J\cap J_\sigma = \emptyset$.

Therefore there exists a basis of solutions $\calB_\sigma$ of the differential system $Y'=AY$ of which the $\roro$ first elements are the $Y_j$, $j\in J$, and the $\Card J_\sigma$ next ones are the $Y_j$, $j\in J_\sigma$. The wronskian determinant of $\calB_\sigma$ may vanish at $\sigma$ if $\sigma$ is a singularity, but even in this case it has generalized order $\leq c_0(\sigma)$ at $\sigma$ (see \cite[\S 3.1]{SFcaract}), where $c_0(\sigma)$ is a constant depending only on $A$ and $\sigma$ (not on $\calB_\sigma$). On the other hand, all components of all elements of $\calB_\sigma$ have generalized order $\geq r_\sigma $ at $\sigma$, where $r_\sigma\in\R$ depends only on $A$ and $\sigma$ (see \cite[Proposition~1]{BB}). 
 Therefore 
 there exists a subset $I_\sigma$ of $\unq$, with $\Card I_\sigma = q- \roro- \Card J_\sigma$, such that the determinant of the submatrix of $T_\sigma$ corresponding to the columns indexed by $\unq \setminus I_\sigma$
has generalized order $\leq c(\sigma)$ at $\sigma$,
 where $c(\sigma) = c_0(\sigma) - r_\sigma \Card I_\sigma$ depends only on $A$ and $\sigma$. Increasing $c_0(\sigma)$ and $c(\sigma)$ if necessary, we may assume that $c(\sigma)\in\N$. 
 
 Let $P_\sigma \in M_q(\hol) $ denote the matrix of which the $\roro+ \Card J_\sigma$ first rows are that of $T_\sigma$, and the other rows are the $\tra e_i$, $i\in I_\sigma$, where $(e_1,\ldots,e_q)$ is the canonical basis of $M_{q,1}(\C)$. Then $P_\sigma M$ has its 
$\roro$ first rows equal to $(0 \,\, 0 \ldots 0)$, its $ \Card J_\sigma$ next ones equal to $ \left( \begin{array}{cccc} R(Y_j) & R(Y_j)' & \ldots & R(Y_j)^{(q-1)}\end{array}\right)$ with $j\in J_\sigma$, and its last rows equal to $ \left( \begin{array}{ccc} P_{1,i} \;\ldots\; P_{q,i}\end{array}\right)$ with $i \in I_\sigma$. Therefore all coefficients in the row corresponding to $j \in J_\sigma$ have order at $\sigma$ at least $\ord_\sigma R(Y_j)-q+1$, and (if $\sigma\neq\infty$) all coefficients in the row corresponding to $i \in I_\sigma$ are either holomorphic at $\sigma$, or have a pole of order at most $(q-1)N_\sigma$ is $\sigma\in S$. 

By construction of $I_\sigma$, the determinant of $P_\sigma$ has generalized order $\leq c(\sigma)$ at $\sigma$; in particular $P_\sigma$ is invertible and we may write $M = P_\sigma^{-1} (P_\sigma M)$. Since the first $\roro$ rows of $P_\sigma M\in M_q(\hol) $ are identically zero, we have $\rk (P_\sigma M) \leq q -\roro$ so that $\rk M \leq q -\roro$. Combining this inequality with Step 1, we obtain $\rk M = q -\roro$.

Recall that $M_0(z)$ is a $\mu\times\mu$ submatrix of $M(z)$ such that $\det M_0(z)\neq 0$, with $\mu = \rk M = q -\roro$. Let $k_1 < \ldots < k_\mu$ denote the indices $k$ such that the $k$-th column of $M(z)$ appears in $M_0(z)$. Denote by $M_1(z) \in M_\mu(\hol) $ the submatrix of $P_\sigma M$ obtained by removing the first $\roro$ rows (which are identically zero) and keeping only the columns with indices $k_1$, \ldots, $k_\mu$. Let $M_2(z) \in M_\mu(\hol) $ denote the submatrix of $P_\sigma ^{-1}$ obtained by removing the first $\roro$ columns, and keeping only the columns with indices $i_1$, \ldots, $i_\mu$ where $i_1<\ldots< i_\mu$ are the indices $i$ such that the $i$-th row of $M(z)$ appears in $M_0(z)$. Then the identity $M = P_\sigma^{-1} (P_\sigma M)$ yields $M_0 = M_2M_1$ because the first $\roro$ rows of $P_\sigma M$ are identically zero. Recall that all coefficients in the row of $M_1$ corresponding to $j \in J_\sigma$ have order at $\sigma$ at least $\ord_\sigma R(Y_j)-q+1$, and (if $\sigma\neq\infty$) all coefficients in the row corresponding to $i \in I_\sigma$ are either holomorphic at $\sigma$, or have a pole of order at most $(q-1)N_\sigma$ is $\sigma\in S$. 
Since $N_\sigma = 0$ if $\sigma\not\in S$, we have if $\sigma\in\Sigma\setminus\{\infty\}$:
$$\ord_\sigma \det( M_1) \geq \Big( \sum_{j\in J_\sigma}\ord_\sigma R(Y_j)\Big) - (q-1)\Card J_\sigma -(q-1)N_\sigma (\mu-\Card J_\sigma).$$

Now let us focus on $M_2$. Recall that all coefficients of the $\roro + \Card J_\sigma$ first rows of $P_\sigma$ have generalized order $\geq r_\sigma$ at $\sigma$, and all coefficients of the other rows are equal to 0 or 1. Therefore all coefficients of the comatrix ${\rm Com} P_\sigma$ have generalized order $\geq (\roro + \Card J_\sigma) r_\sigma$ at $\sigma$ (because we may assume that $r_\sigma\leq 0$). Since $\ord_\sigma \det( P_\sigma)\leq c(\sigma)$ and $M_2$ is a submatrix of $P_\sigma^{-1} = (\det P_\sigma)^{-1} \, \, \tra \, {\rm Com} P_\sigma$, we deduce that all coefficients of $M_2$ have generalized order $\geq (\roro + \Card J_\sigma) r_\sigma-c(\sigma)$ at $\sigma$, so that 
\begin{equation} \label{eqdetmde}
\ord_\sigma \det(M_2) \geq \mu (\roro + \Card J_\sigma) r_\sigma - \mu c(\sigma).
\end{equation}
Combining this inequality with the corresponding one on $M_1$, we obtain
\begin{equation} \label{eqdetmtr}
\ord_\sigma \det( M_0) \geq \Big( \sum_{j\in J_\sigma}\ord_\sigma R(Y_j)\Big) - c'(\sigma) 
\end{equation}
for some integer constant $c'(\sigma) $ which depends only on $A$ and $\Sigma$. Now let
$$Q_2(z) = \Big( \prod_{s\in S} (z-s)^{q(q-1)N_s} \Big) \cdot \Big( \prod_{\sigma\in\Sigma\setminus\{\infty\}}(z-\sigma)^{c'(\sigma)}\Big)$$
so that $Q_2(z) \det M(z)$ is a polynomial and vanishes at any $\sigma \in\Sigma\setminus\{\infty\}$ with order at least $ \sum_{j\in J_\sigma}\ord_\sigma R(Y_j)$. To bound from above the degree of this polynomial, we define $T_\infty$ and $P_\infty$ as above; 
 if $\infty\not\in\Sigma$ then $J_\infty=\emptyset$ and there is no row corresponding to $J_\infty$. Everything works in the same way as for $\sigma\in\C$, except that for some non-negative integer $t$ we have $R(Y_j)^{(k-1)} = O(z^{-\ord_\infty R(Y_j)}\log (z)^t)$ as $|z|\to\infty$ for any $j\in J_\infty$ and any $k\geq 1$, and $P_{k,i}(z) = O(z^{n+(q-1)d})$ for any $i\in I_\infty$ and any $k \in \unq$ (where $d$ is greater than or equal to the degree of all coefficients of $A$). Therefore we have $\ord_\infty R(Y_j)^{(k-1)} \geq \ord_\infty R(Y_j)$ for any $j\in J_\infty$ and any $k\geq 1$, and 
 $\ord_\infty P_{k,i} \geq -n-(q-1)d$ for any $i\in I_\infty$ and any $k \in \unq$. Arguing as above we obtain 
 $$
 \ord_\infty \det( M_1) \geq \Big( \sum_{j\in J_\infty}\ord_\infty R(Y_j)\Big) - (n+(q-1)d)(\mu- \Card J_\infty),$$
and Eq. \eqref{eqdetmde} remains valid with $\sigma=\infty$, so that 
 $$
 \ord_\infty \det( M_0) \geq \Big( \sum_{j\in J_\infty}\ord_\infty R(Y_j)\Big) - n(\mu- \Card J_\infty) - c'(\infty)$$
for some $c'(\infty) $ which depends only on $A$ and $\Sigma$. 
 Since $\det M_0(z)$ is a rational function this means $\det M_0(z) = O(z^u)$ as $|z| \to \infty$, with
$$u = c'(\infty) + n(q-\roro - \Card J_\infty) - \sum_{j\in J_\infty} \ord_{\infty}R(Y_j),$$
so that
$$\deg ( Q_2(z) \det M_0(z)) \leq u + \deg Q_2 \leq \tau + \cstun + \sum_{\sigma\in\Sigma\setminus \{\infty\} } \Big\lceil\, \sum_{j\in J_\sigma} \ord_{\sigma}R(Y_j)\Big\rceil $$
using Eq. \eqref{eqhypdetnn}, where $\cstun$ depends only on $A$ and $\Sigma$ (since $0\leq \Card J_\sigma \leq q$ for any $\sigma$) and $\lceil \omega \rceil$ is the least integer greater than or equal to $\omega$. 
 To sum up,
 we have found a polynomial $Q_1$ of degree at most $\tau + \cstun$ such that
$$\det M_0(z) = \frac{Q_1(z)}{Q_2(z)} \prod _{\sigma\in\Sigma\setminus \{\infty\}} (z-\sigma)^{\lceil\, \sum_{j\in J_\sigma} \ord_{\sigma}R(Y_j) \rceil }.$$
 Since $ \det M_0\neq 0$, this implies in particular $\tau\geq -\cstun$.

\bigskip

\noindent {\bf Step 3: } Evaluation at $\alpha$.

\medskip

To begin with, we denote by $\loga$ the $\C$-vector space spanned by the functions $h(z) (z-\alpha)^e (\log (z-\alpha))^i$ with $h$ holomorphic at $\alpha$, $i\in\N$, and $e\in\C$ such that either $e=0$ or $\Re(e) > 0$.

Let $q_\alpha = \Card J_\alpha$ and $q'_\alpha = q-\roro-q_\alpha$, where $J_\alpha = \emptyset $ if $\alpha\not\in\Sigma$; for simplicity we assume that $ J = \{1,\ldots,\roro\}$ and $J_\alpha = \{\roro+1,\ldots,\roro+q_\alpha\}$. Since the solutions $Y_{1}$, \ldots, $Y_{\roro+q_\alpha}$ of the differential system $Y'=AY$ are linearly independent over $\C$ (see Step 2), there exist solutions $Y_{\roro+q_\alpha+1}$, \ldots, $Y_q$ such that $(Y_1,\ldots, Y_q)$ is a local basis of solutions at $\alpha$. Let $\calY $ be the matrix with columns $Y_1$, \ldots, $Y_q$; then $\tra \, \calY M$ is the matrix 
$[R(Y_i)^{(k-1)}]_{1\leq i,k\leq q}$, and assumption $(i)$ of Theorem \ref{thshid} yields $\calY \in M_q(\loga)$. Moreover the first $\roro$ rows of $ \tra \, \calY M$ are identically zero.

\bigskip

As in Step 2 we fix a square submatrix $M_0(z)$ of $M(z)$ of size $\mu=q-\roro$, such that $\det M_0(z)\neq 0$. We also denote by $k_1 < \ldots < k_{q-\roro}$ the indices $k$ such that the $k$-th column of $M(z)$ appears in $M_0(z)$. We shall consider now the matrix $M'_0\in M_{q-\roro}(\hol)$ obtained from $ \tra \, \calY M$ by removing the first $\roro$ rows (which are identically zero) and keeping only the columns with indices $k_1$, \ldots, $k_{q-\roro}$. In other words, we have $M'_0(z) = \big(R(Y_{\roro+i})^{(k_j-1)}\big)_{1\leq i,j \leq q-\roro}$.

\bigskip

Let $Z$ denote the submatrix of $ \tra \, \calY^{-1}$ obtained by keeping only the columns with indices $\roro+1$, \ldots, $q$ and the rows with indices $i_1$, \ldots, $i_{q-\roro}$ (where $1\leq i_1 < \ldots < i_{q-\roro}\leq q$ are the indices of the rows of $M$ that appear in $M_0$). Then we have $ZM'_0 = M_0$. 
 Moreover $\det \calY $ is the wronskian of $Y_1$, \ldots, $Y_q$: it is a solution of the first order differential equation
\begin{equation} \label{eqwr}
w'(z) = w(z) {\rm trace}(A(z)) .
\eneq
Now the generalized order at $\alpha$ of any non-zero solution of Eq. \eqref{eqwr}, and in particular that of $\det \calY$, can be bounded from above in terms of $A$ only. 
On the other hand, as in Step 2, all coefficients of $\calY$ have generalized order $\geq r_\alpha$ at $\alpha$. 
Since $\tra \,\calY ^{-1} = (\det \calY)^{-1}\, \tra\, {\rm Com}(\tra\, \calY)$, all coefficients of $\tra\, \calY^{-1}$, and accordingly $\det Z$, have generalized order at $\alpha$ bounded from below in terms of $A$ only. On the other hand, Step 2 shows that $\ord_\alpha\det M_0 \leq \omega_\alpha + \tau + \cstun$, where $\omega_\alpha = \lceil \sum_{j\in J_\alpha} \ord_\alpha R(Y_j)\rceil$ if $\alpha\in\Sigma$, and $\omega_\alpha =0$ otherwise. Finally, the matrix $M'_0 = Z^{-1}M_0$ satisfies 
\begin{equation} \label{eqmajoordmpr}
\ord_\alpha\det M'_0 \leq \omega_\alpha + \tau + \cstunnv
\end{equation}
where $\cstunnv$ depends only on $A$.

\bigskip

For any subset $\AAA$ of $\unqmro$ of cardinality $q'_\alpha := q-\roro-q_\alpha$, we denote by $\Delta_\AAA$ the determinant of the submatrix of $M'_0$ obtained by considering only the rows with index $i\geq q_\alpha+1$ and the columns with index $j\in \AAA$, and by $\Deltati_\AAA$ the one 
obtained by removing these rows and columns. Then Laplace expansion by complementary minors yields
\begin{equation} \label{eqLaplace}
 \det M'_0(z) = \sum_{\AAA \subset \unqmro \atop \Card \AAA = q'_\alpha} \eps_\AAA \Delta_\AAA(z)\Deltati_\AAA(z)
\end{equation}
with $\eps_\AAA\in\{-1,1\}$.
Using Eq. \eqref{eqmajoordmpr} there exists a subset $\AAA$ such that 
\begin{equation} \label{eqmajootr}
\otr := \ord_\alpha \Delta_\AAA(z) \leq \omega_\alpha + \tau - \cstunnv - \ord_\alpha \Deltati_\AAA(z) .
\end{equation} 
Now for any $i \in J_\alpha= \{\roro+1,\ldots,\roro+q_\alpha\}$ and any $k\in\unq$ we have $\ord_\alpha R(Y_i)^{(k-1)}\geq \ord_\alpha R(Y_i) - (q-1)$ so that 
 $\ord_\alpha \Deltati_\AAA(z) \geq \omega_\alpha - q_\alpha(q-1)$. Therefore Eq. \eqref{eqmajootr} yields
$\otr \leq \tau + \csttr$ for some constant $\csttr$ depending only on $A$ and $\Sigma$. 
Using this upper bound we shall prove now that $\otr$ is a non-negative integer, and $\Delta_\AAA ^{(\otr)}(z) $ has a finite non-zero limit as $z$ tends to $\alpha$. 

\bigskip

Since $\calY \in M_q(\loga)$ and $P_{i,k}$ has no pole at $\alpha$ for $k\leq q$, we have $\Delta_\AAA (z) \in \loga $ so that 
\begin{equation} \label{eqdeltae}
\Delta_\AAA (z) = \sum_{e\in\cale} \sum_{i=0}^I \lambda_{i,e} \, h_{i,e}(z) (z-\alpha)^e ( \log(z-\alpha))^i
\end{equation}
where $h_{i,e}(z) $ is holomorphic at $\alpha$ and $\cale$ is a finite subset of $\C$ such that for any $e\in\cale$, either $e=0$ or $\Re(e)>0$. Moreover we may assume that $e-e'\not\in\Z$ for any distinct $e,e'\in\cale$, and that for any $e\in\cale$ there exists $i$ such that $\lambda_{i,e} h_{i,e}(\alpha)\neq 0$. At last, the integer $I$ can be chosen in terms of $A$ only, since the exponents of $\log(z-\alpha)$ in local solutions at $\alpha$ of $Y' = AY$ are bounded. 

We choose the constant $\cstnum$ of Theorem \ref{thshid} to be $\cstnum = \csttr + I + q+1$. For any non-negative integer $\varpi \leq \cstnum +\tau - q - 1$, the $\varpi$-th derivative $\Delta_\AAA ^{(\varpi)}(z) $ is a $\Z$-linear combination of determinants of matrices of the form
$$N_{k'_1,\ldots,k'_{q'_\alpha}} = \big(R(Y_{\roro+q_\alpha + i})^{(k'_j-1)}\big)_{1\leq i,j \leq q'_\alpha}$$
with $1\leq k'_1 < \ldots < k'_{q'_\alpha} \leq q + \varpi < \cstnum +\tau $. Since $Y_i \in M_{q,1}(\loga)$ and $P_{k,i}$ is assumed to be holomorphic at $\alpha$ for any $i$ and any $k< \tau + \cstnum$, we have $
R(Y_i)^{(k-1)} \in \loga$. Accordingly $\det N_{k'_1,\ldots,k'_{q'_\alpha}} \in \loga$, and finally $\Delta_\AAA ^{(\varpi)}(z) \in\loga$ for any non-negative integer $\varpi \leq \cstnum +\tau - \roro n - q - 1$. Therefore in the expression \eqref{eqdeltae}, all pairs $(e,i)$ such that 
$\lambda_{i,e} h_{i,e}(\alpha)\neq 0$ and 
$\Re(e)+i \leq \cstnum +\tau - q - 1$ satisfy $e\in\N$ and $i=0$. Now recall that $\otr = \ord_\alpha \Delta_\AAA(z) \leq \tau + \csttr = \cstnum +\tau - q - 1 - I$. Then there is a term $(e,i)$ in Eq. \eqref{eqdeltae} such that $\lambda_{i,e} h_{i,e}(\alpha)\neq 0$ and $\Re(e) = \otr \leq \cstnum +\tau - q - 1 - I$, and accordingly $
\Re(e)+i \leq \cstnum +\tau - q - 1$: we have $e\in\N$, $i=0$, and no other term $(e',i')$ such that $\lambda_{i',e'} h_{i',e'}(\alpha)\neq 0$ satisfies $\Re(e') = \otr$. In particular $\otr$ is a non-negative integer, and $\Delta_\AAA ^{(\otr)}(z) $ has a finite non-zero limit as $z$ tends to $\alpha$.

\bigskip

Let $\evalpha : \loga\to\C$ denote the regularized evaluation at $\alpha$, defined by $\evalpha (f) = \lambda_{0,0} \, h_{0,0}(\alpha)$ if $f(z)$ is the right hand side of Eq. \eqref{eqdeltae}, and of course $\evalpha (f) = 0$ if $0\not\in\cale$.
 The important point here is that any $e\in\cale$ satisfies either $e=0$ or $\Re(e)>0$, so that $\evalpha$ is a $\C$-algebra homomorphism; moreover $\evalpha(f)$ is equal to the limit of $f(z)$ as $z\to\alpha$ whenever this limit exists. In particular we have $
\evalpha( \Delta_\AAA ^{(\otr)})\neq 0$. Now, as above $\evalpha ( \Delta_\AAA ^{(\otr)}) $ is a $\Z$-linear combination of $\evalpha (\det N_{k'_1,\ldots,k'_{q'_\alpha}} )$ with $1\leq k'_1 < \ldots < k'_{q'_\alpha} \leq q + \otr < \cstnum +\tau $, so that $\evalpha (\det N_{k'_1,\ldots,k'_{q'_\alpha}} )\neq 0 $ for some tuple $(k'_1 , \ldots, k'_{q'_\alpha} )$. 
For this tuple we consider the equality $\tra \, \calYti \Mti = 
 N_{k_1,\ldots,k_{q'_\alpha}}$, where $\calYti \in M_{q,q'_\alpha}(\loga)$ is the matrix with columns $Y_{\roro+q_\alpha+1}$, \ldots, $Y_q$, and $\Mti = \big(P_{k'_j,i}\big)_{1\leq i \leq q, 1\leq j \leq q'_\alpha}$. The Cauchy-Binet formula yields
\begin{equation} \label{eqCB1} 
\det N_{k'_1,\ldots,k'_{q'_\alpha}} = \sum_{B \subset \unq \atop \Card B = q'_\alpha} \det {} \tra \, \calYti_B \cdot \det \Mti_B
\end{equation}
where $\calYti_B$ (resp. $\Mti_B$) is the square matrix consisting in the rows of $\calYti$ (resp. of $\Mti$) corresponding to indices in $B$.
 Extending $\evalpha$ coefficientwise to matrices, Eq. \eqref{eqCB1} yields
$$\evalpha \Big(\det N_{k'_1,\ldots,k'_{q'_\alpha}}\Big) = \sum_{B \subset \unq \atop \Card B = q'_\alpha} \evalpha \Big(\det{} \tra\, \calYti_B \Big)\cdot \evalpha \Big(\det \Mti_B\Big).$$
Now the left hand side is non-zero, so that $\evalpha (\det \Mti_B) \neq 0$ for some $B$. Since all coefficients $P_{k,i}$ are holomorphic at $\alpha$, so is $\det \Mti_B $ and therefore $\det (\Mti_B(\alpha) )= \evalpha (\det \Mti_B) \neq 0$. We have found an invertible submatrix of $M(\alpha)$ of size $q'_\alpha$, so that $\rk (M(\alpha)) \geq q'_\alpha$: this concludes the proof of Theorem \ref{thshid}.

\section{Linear independence of the linear forms} \label{secapplishid}

In this section we combine Theorems \ref{thpade} and \ref{thshid} to construct linearly independent linear forms (that will be used in \S \ref{secfin}, together with Siegel's linear independence criterion, to prove Theorem \ref{thintroun}). Our main result is Proposition \ref{propshid} that we shall state now, and prove in \S \S \ref{subsecshiddebut} and \ref{subsecshidfin} using an explicit computation of $J_f(z)$ (see \S \ref{ssec:expcompJf3110}).

\bigskip

Let 
$$I = \Big( \{1,\ldots,\lun\}\times \{1,\ldots, S\} \Big) \, \, \sqcup \, \, \{0,\ldots,\mu-1\} $$
and $q = \Card \, I = \lun S+\mu $. Elements of $I$ will be denoted by $(u,s)$ (with $1\leq u \leq \lun$ and $1\leq s \leq S$) or $u$ (with $0\leq u \leq \mu-1$). For any $n$ sufficiently large, Lemma \ref{lem600} provides a family $(\bfP_i)_{i\in I}$ of polynomials indexed by $I$, namely $\bfP_{u,s} = P_{u,s,n}$ and $\bfP_u = \Pti_{u,n}$; here the integer $n$ is omitted in the notation. 

Let us denote by $\calW$ the set of tuples
$$Y = (y_i)_{i\in I} = \Big( (y_{u,s})_{1\leq u \leq \lun, 1\leq s \leq S}, \, (\yti_u)_{0\leq u \leq \mu-1} \Big)$$
consisting in functions $y_{u,s}$ and $\yti_u$ holomorphic on $\dpr$ that obey the same differentiation rules as if they were given by $y_{u,s} = y_u^{[s]}$ and $\yti_u = \theta^u y$ with $y\in\ker L$; recall that $\dpr$ has been defined in \S \ref{subsecnotationspade}.
 In precise terms we require:
$$\left\{ 
\begin{array}{rcl}
y'_{u,s}(z) &=& \frac1{z} y_{u,s-1}(z) \mbox{ for any $1\leq u \leq \lun$ and $2\leq s \leq S$}\\
\\
y'_{u,1}(z) &=& z^{u-1} \yti_{0}(z) \mbox{ for any $1\leq u \leq \lun$} \\
\\
 \yti \hspace{0.3mm} ' _u (z) &=& \frac1{z} \yti_{u+1}(z) \mbox{ for any $0\leq u \leq \mu-2$}\\
 \\
L\yti_0 &=& 0 
\end{array}
\right.$$
Since $L$ has order $\mu$, there exist $R_0, \ldots, R_\mu\in\K(z)$ such that $L= \sum_{u=0}^{\mu } R_u(z) \theta^u$ and $R_\mu\neq 0$ ; the equation $L\yti_0 =0$ can be replaced with
$$\yti \hspace{0.3mm} ' _{\mu-1} (z) = \frac{-1}{z R_\mu(z)} \sum_{u=0}^{\mu-1} R_u(z) \yti_u(z).$$
We obtain in this way a square matrix $A$ of size $q$, with rows and columns indexed by $I$ and coefficients in $\K(z)$, such that $\calW$ is exactly the set of solutions holomorphic on $\dpr$ of the differential system $Y'=AY$.
Here and below, when $Y = (y_i)_{i\in I}$ is an element of $\calW$, we shall consider $Y$ as a column vector and let (as in \S \ref{secshid})
$$R(Y)(z) = \sum_{i\in I} \bfP_i(z) y_i(z) = 
\sum_{u=1}^{\lun}\sum_{s=1}^S P_{u,s,n}(z) y_{u,s}(z)+ \sum_{u=0}^{\mu-1} \Pti_{u,n}(z) \yti_u (z) .$$
The point is that if $y_{u,s} = y_u^{[s]}$ and $\yti_u = \theta^u y$ for some $y\in\ker L$, then $R(Y)(z)=J_y(z)$ with the notation of Theorem \ref{thpade}, and Eq. \eqref{eqderiR} yields
\begin{equation} \label{eqjder}
J_y^{(k-1)}(z) = R(Y)^{(k-1)}(z) = \sum_{i\in I} \bfP_{k,i}(z) y_i(z) \mbox{ for any $k\geq 1$ and any $z\in\dpr$.}
\end{equation}
In the following Proposition the $\bfP_{k,i}(z)\in \K(z)$ are evaluated at a point $z_0 \in\dpr\setminus\{0\}$, so that $z_0$ is not a singularity of $L$ or $A$, and accordingly not a pole of any of these rational functions (cf. Eq. \eqref{eqdefpki}).

\begin{Prop} \label{propshid} Under the assumptions of Theorem \ref{thpade}, suppose also that $r\geq 1$ and that $n$ is sufficiently large. Put $\roro = \dim( \cz\cap\ker L)$ and let $z_0 \in\dpr\setminus\{0\}$. Then:
\begin{enumerate}
\item[$(i)$] There exist pairwise distinct elements $i_1$, \ldots, $i_\roro$ of $I$ such that
$$\bfP_{k,i_t} (z_0) = \sum_{i \in I \setminus \{i_1,\ldots,i_\roro\}} \lambda_{i,t} \bfP_{k,i}(z_0) \mbox{ for any } t\in\{1,\ldots,\roro\} \mbox{ and any } k\geq 1
$$
with $\lambda_{i,t}\in\K$; here $i_1$, \ldots, $i_\roro$ and the $\lambda_{i,t}$ depend only on $L$ and $z_0$ but not on $k$.
\item[$(ii)$] There exist integers $1 \leq k_1 < k_2 < \ldots < k_{q-\roro}$, bounded from above in terms of $L$, $r$, $S$ only, such that the matrix $(\bfP _{k_j,i}(z_0))_{i\in I, \, 1\leq j \leq q-\roro}$ has rank $q-\roro$.
\end{enumerate}
\end{Prop}

Of course this proposition shows that the matrix $(\bfP _{k_j,i}(z_0))_{i\in I\setminus \{i_1,\ldots,i_\roro\}, \, 1\leq j \leq q-\roro}$
 is invertible; in the proof of Theorem \ref{thintroun} we shall apply Siegel's linear independence criterion to this matrix (see \S \ref{secfin}). 
 
 \bigskip

The rest of this section is devoted to the proof of Proposition \ref{propshid}. We begin with a technical lemma.

\subsection{Explicit computation of $J_f(z)$}\label{ssec:expcompJf3110}

In \S \ref{subsecshidfin} we shall use the following technical lemma, which gives an explicit expression of $J_f(z)$ (see Remark \ref{remsanslog} for an easier special case).

\begin{lem} \label{lemexplicite}
Let $f\in \mathcal N$ belong to the Nilsson class with rational exponents at 0; write
$$f(z) = \sum_{k\in \Q \atop k\geq \capaf} \sum_{i=0}^{e} a_{k,i} z^k \log (z)^i.$$
 Then we have
\begin{eqnarray*}
J_f(z) 
&=& \sum_{k\in \Q \atop k\geq \capaf} z^{n+1+k} \sum_{\lambda=0}^e \log (z)^\lambda \sum_{i=\lambda}^e a_{k,i} \binom{i}{\lambda} A^{(i-\lambda)}(k) \\
&+& \sum_{s=1}^S \sum_{i=0}^e \log (z)^{s+i} \sum_{j=1}^{n+1} z^{n+1-j} a_{-j,i} c_{j,s,n} \frac{i!}{(s+i)!}
\end{eqnarray*}
where $A^{(i-\lambda)}(k) $ is the $(i-\lambda)$-th derivative of $A(X) = n!^{S-r} \frac{ (X-rn+1)_{rn}}{(X+1)_{n+1}^S}$ taken at $X=k$ if $k\not\in\{-1,\ldots, -n-1\}$; the general definition of this number (valid for any $k\in\Q$) is given by
$$A^{(i-\lambda)}(k) = (i-\lambda)! \sum_{j=1 \atop j\neq -k }^{n+1} \sum_{s=1}^S c_{j,s,n} \frac{(-1)^{i-\lambda}\binom{s-1+i-\lambda}{s-1}}{(k+j)^{s+i-\lambda}}.$$
\end{lem}

\begin{proof}
To begin with, recall that in the proof of Lemma \ref{lemvraiesol} we have obtained the following identity:
$$
 J_f (z)=\sum_{j=1}^{n+1} \sum_{s=1}^S c_{j,s,n} z^{n+1-j}f_{j}^{[s]}( z).
$$
Using Lemma \ref{lemexplicitefns} we obtain
\begin{eqnarray*}
J_f(z) 
&=& \sum_{j=1}^{n+1} \sum_{s=1}^S c_{j,s,n} z^{n+1-j}
 \sum_{i=0}^e a_{-j, i} i! \frac{\log (z)^{s+i}}{(s+i)! } \\
&&+ \sum_{j=1}^{n+1} \sum_{s=1}^S c_{j,s,n} 
\sum_{k\in \Q\setminus\{-j\} \atop k\geq \capaf} z^{n+1+k} \sum_{\lambda=0}^{e} \frac{\log (z)^\lambda}{\lambda!} \sum_{i=\lambda}^e a_{k,i} (-1)^{i-\lambda} \frac{i! \binom{s-1+i-\lambda}{s-1} }{(k+j)^{s+i-\lambda}} \\
&=& \sum_{s=1}^S \sum_{i=0}^e \log (z)^{s+i} \sum_{j=1}^{n+1} z^{n+1-j} a_{-j, i} c_{j,s,n} \frac{i!}{(s+i)! }\\
&&+ \sum_{k\in \Q \atop k\geq \capaf} z^{n+1+k} \sum_{\lambda=0}^{e} \frac{\log (z)^\lambda}{\lambda!}\sum_{i=\lambda}^e a_{k,i} i! \sum_{j=1 \atop j\neq k }^{n+1} \sum_{s=1}^S c_{j,s,n} (-1)^{i-\lambda} \frac{\binom{s-1+i-\lambda}{s-1} }{(k+j)^{s+i-\lambda}} .
\end{eqnarray*}
This concludes the proof of Lemma \ref{lemexplicite}.
\end{proof}

\bigskip

\begin{coro} \label{corbourrin}
With the notations of Lemma \ref{lemexplicite}, assume that $S>e$ and
\begin{equation} \label{eqbourrin}
J_f(z) = \sum_{p=0}^n \sum_{s = 0}^{S-1} \alpha_{p,s} z^p \log (z)^s \mbox{ with } \alpha_{p,s} \in\C.
\end{equation} 
Then $f$ is a polynomial.
\end{coro}

\begin{Remark} Of course if $f$ is a polynomial then so is $f_j^{[s]}$ for any $j,s$, and therefore also $J_f(z)$.
\end{Remark}

\begin{proof}
In this proof we keep the notation of Lemma \ref{lemexplicite}, that we shall use repeatedly to compute the coefficient of $z^\tau \log (z)^\sigma$ in Eq. \eqref{eqbourrin} for various pairs $(\tau,\sigma)\in\Q\times\N$. The point is that this coefficient is zero if $\tau\not\in\{0,\ldots,n\}$ or $\sigma\not\in\{0,\ldots,S-1\}$.

To begin with, let $k\in\Q\setminus\{-n-1,\ldots,rn-1\}$. For any $\sigma\in\{0,\ldots,e\}$, considering the coefficient of $z^\tau \log (z)^\sigma$ with $\tau = k+n+1\not\in\{0,\ldots,n\}$ yields $\sum_{i=\sigma}^e a_{k,i} \binom{i}{\sigma} A^{(i-\sigma)}(k)=0$. Now $k \not\in\{0, \ldots,rn-1\}$ so that $A(k)\neq 0$; therefore by decreasing induction we obtain $a_{k,i}=0$ for any $i\in \{0,\ldots,e\}$.

Now let $k\in \{-n-1,\ldots,-1\}$. For any $\sigma\in\{S,\ldots,S+e\}$ we consider the coefficient of $z^\tau \log (z)^\sigma$ with $\tau = k+n+1$. Since $\sigma\geq S > e$ this yields $\sum_{i=\sigma-S}^e a_{k,i} c_{-k, \sigma-i, n} \frac{i!}{\sigma!}=0$. Using this relation with $\sigma = S+e$ we obtain $a_{k,e} = 0$ since $ c_{-k, S, n} \neq 0$. By decreasing induction we prove that $a_{k,i}=0$ for any $i\in \{0,\ldots,e\}$.

At last we take $k\in \{0,\ldots,rn-1\}$. Since $A(k)=0$, considering the coefficient of $z^{k+n+1}\log (z)^\sigma$ yields $\sum_{i=\sigma+1}^e a_{k,i} \binom{i}{\sigma} A^{(i-\sigma)}(k)=0$ for any $\sigma\in\{0,\ldots,e-1\}$. By decreasing induction this implies $a_{k,i}=0$ for any $i\in \{1,\ldots,e\}$ because $A'(k)\neq 0$.

In conclusion we have proved that $a_{k,i}=0$ for any pair $(k,i)\in \Q \times \{0,\ldots,e\}$, except maybe when $k\in \{0,\ldots,rn-1\}$ and $i=0$. Therefore $f$ is a polynomial of degree less than $rn$.
\end{proof}

\subsection{Proof of Proposition \ref{propshid} $(i)$: polynomial solutions} \label{subsecshiddebut}

In this section we focus on polynomial solutions of $L$, and prove that each one provides a solution of the differential system $Y'=AY$ such that $R(Y)=0$. This will enable us to prove part $(i)$ of Proposition \ref{propshid}.

\bigskip

Recall that $\roro = \dim (\cz\cap\ker L)$, and choose a basis $( f^{\{1\}},\ldots, f^{\{\roro\}})$ of $\cz\cap\ker L$. For any $\ppp\in\{1,\ldots,\roro\}$ we let 
$$\Yde = \Big( ( (f^{\{\ppp\}})_u^{[s]}) _{1\leq u \leq \lun, 1\leq s \leq S}, \, (\theta^u f^{\{\ppp\}}) _{0\leq u \leq \mu-1} \Big).$$
 Then we have $\Yde\in\calW$, and $R(\Yde)(z) = J_{f^{\{\ppp\}}}(z)$. Using Eq. \eqref{eqpadehol} we obtain 
 $\ord_0 R(\Yde)\geq (r+1)n +1$ for any $\ppp\in\{1,\ldots,\roro\}$. Now 
 $f^{\{\ppp\}}$ is a polynomial and its degree is bounded in terms of $L$ only. Therefore Eq. \eqref{eqJholo} shows that $R(\Yde) = J_{f^{\{\ppp\}}} $ is a polynomial of degree at most $n+c$, where $c$ depends only on $L$. Since $r\geq 1$ and $n$ is large enough, we deduce that $J_{f^{\{\ppp\}}} = R(\Yde)$ is the zero polynomial for any $\ppp\in\{1,\ldots,\roro\}$. Therefore Eq. \eqref{eqjder} 
shows that for any $k\geq 1$, the values $x_i = \bfP_{k,i}(z_0)$ make up a solution of the following linear system:
 \begin{equation} \label{eqsyslin}
 \sum_{i\in I} \petityde_i(z_0) x_i = 0 \mbox{ for any } \ppp \in \{1,\ldots, \roro\}.
 \end{equation}
Now $z_0 \in\dpr\setminus\{0\}$ so that $z_0$ is not a singularity of $L$, and not of the differential system $Y'=AY$ either. Therefore the map $\calW\to \C^I$, $Y = (y_i) \mapsto (y_i(z_0))_{i\in I}$ is bijective. Since $\Ydeun$, \ldots, $\Yderoro$ are $\C$-linearly independent (because $f^{\{1\}}$, \ldots, $f^{\{\roro\}}$ are and $\theta^0 f^{\{\ppp\}} = f^{\{\ppp\}}$ is a component of $\Yde$), we deduce that the matrix $[ \petityde_i(z_0) ]_{i\in I, 1\leq \ppp \leq \roro}$ of the linear system \eqref{eqsyslin} has rank $\roro$. This linear system can therefore be put in reduced row-echelon form as follows: there exist pairwise distinct elements $i_1$, \ldots, $i_\roro$ of $I$, and coefficients $\lambda_{i,t}\in\K$, such that this system is equivalent to 
$$ x_{i_t} = \sum_{i \in I \setminus \{i_1,\ldots,i_\roro\}} \lambda_{i,t} x_i \mbox{ for any } t \in \{1,\ldots, \roro\}.
$$
Since $x_i = \bfP_{k,i}(z_0)$ is a solution of this linear system for any $k\geq 1$, this concludes the proof of part $(i)$ of Proposition \ref{propshid}.

\subsection{Proof of Proposition \ref{propshid} $(ii)$} \label{subsecshidfin}

In this section we check the assumptions of Theorem \ref{thshid} to apply Shidlovsky's lemma to the solution of the Pad\'e approximation problem given by Theorem \ref{thpade}, thereby proving assertion $(ii)$ of Proposition \ref{propshid}. 

\bigskip

The notation of the present section is the same as those of Sections \ref{secpadestatement} and \ref{secpadepreuve}. It is consistent with the one of \S \ref{subsecsetting}, except for the following. We fix a bijective map $I \to \{1,\ldots,q\}$ so that the family $(\bfP_i)_{i\in I}$ of polynomials involved in Theorem \ref{thpade} can be written as $(P_1,\ldots,P_q)$. The integer $n$ of \S \ref{subsecsetting}, which is an upper bound on $\deg P_i$, is taken equal to $ n+1+S(\ell-1)$. The finite subset denoted by $\Sigma$ in 
\S \ref{subsecsetting} is $\Sigma\cup\{0\}$, where $\Sigma$ is the set of finite singularities of $L$. The family of solutions of the differential system $Y'=AY$ associated with each element of $\Sigma\cup\{0\}$ will be defined below. We let $J = \{1,\ldots,\roro\}$, where $\roro = \dim(\cz\cap\ker L)$. We have constructed in \S \ref{subsecshiddebut} linearly independent solutions $\Yde$ of the differential system $Y'=AY$, for $1 \leq \ppp \leq \roro$, such that $R(\Yde)=0$ for any $\ppp \in \{1,\ldots,\roro\}$.
 We shall apply Theorem \ref{thshid} with $\alpha = z_0$. Since $z_0 \in\dpr\setminus \{0\}$, it is not a singularity of the differential system $Y'=AY$ so that assumptions $(i)$ and $(iii)$ of Theorem~\ref{thshid} hold immediately. Moreover the conclusion of Theorem \ref{thshid} is exactly that of part $(ii)$ of Proposition \ref{propshid}. Therefore to conclude the proof it is enough to check that assumption $(ii)$ of Theorem \ref{thshid} holds; this is what we shall do now. We refer to \S \ref{subseccompte} for a more or less informal presentation of the following ideas.

\bigskip

To begin with, let us consider the vanishing conditions at a non-zero singularity $\alpha\in\Sigma\setminus\{0\}$.
Recall from \S \ref{subseccompte} that $m_\alpha = \dim ( \ker L / ( \hal \cap \ker L))$ is the multiplicity of $\alpha$ as a singularity of $L$. Let $(f^{\{\alpha, 1\}}, \ldots, f^{\{\alpha, m_\alpha\}} )$ be a basis of $ \ker L / ( \hal \cap \ker L)$. For any $1\leq \ppp\leq m_\alpha $, let 
$$Y_{\alpha, \ppp} = \Big( ( \omal ( (f^{\{\alpha, \ppp\}}) ^{[s]}_{u}))_{1\leq u \leq \lun, 1\leq s \leq S}, \, ( \omal( \theta^u f^{\{\alpha, \ppp\}}))_{0\leq u \leq \mu-1} \Big) \in \calW; $$
in other words, $( Y_{\alpha, \ppp} )_{u,s} = \omal ( (f^{\{\alpha, \ppp\}}) ^{[s]}_{u}) $ for any $u,s$, and $( Y_{\alpha, \ppp} )_{u} = \omal( \theta^u f^{\{\alpha, \ppp\}})$ for any $u$. Then Eq. \eqref{eqpadealpha} of Theorem \ref{thpade} reads $R( Y_{\alpha, \ppp} )(z) = O((z-\alpha)^{(S-r)n-\kappa})$ as $z\to \alpha$, that is 
\begin{equation} \label{eqannulalp} 
\ord_{\alpha} R( Y_{\alpha, \ppp} ) \geq (S-r)n-\kappa \mbox{ for any } 1\leq \ppp \leq m_\alpha.
 \end{equation}

Let us prove that $R( Y_{\alpha, 1})$, \ldots, $ R( Y_{\alpha, m_\alpha})$ are linearly independent over $\C$. Let $\lambda_1$, \ldots, $\lambda_{m_\alpha} \in \C$ be such that $\lambda_1 R( Y_{\alpha, 1}) + \ldots + \lambda_{m_\alpha} R( Y_{\alpha, m_\alpha}) = 0$, and put $f = \lambda_1 f^{\{\alpha, 1\}} + \ldots + \lambda_{m_\alpha} f^{\{\alpha, m_\alpha\} }$. Then we have
$$\omal(J_f)(z) = \sum_{u=1}^{\lun}\sum_{s=1}^S P_{u,s,n}(z) \omal ( f ^{[s]}_{u})(z)+ \sum_{u=0}^{\mu-1} \Pti_{u,n}(z) \omal( \theta^u f)(z) =0 .$$
 As asserted in Theorem~\ref{thpade}~$(ii)$, this implies $f$ holomorphic at $\alpha$, i.e. $f=0$ in the quotient space $ \ker L / ( \hal \cap \ker L)$, 
 so that $\lambda_1 = \ldots = \lambda_{m_\alpha} = 0$. This concludes the proof that $R( Y_{\alpha, 1})$, \ldots, $ R( Y_{\alpha, m_\alpha})$ are linearly independent over $\C$. 

\bigskip

Let us move now to the conditions around $z=0$, namely parts $(i)$ and $(iii)$ of Theorem~\ref{thpade}, starting with $(iii)$. Given $u_0 \in \{1,\ldots,m-1\}$ and $s_0 \in \unS$ we define $\Yun$ by
$$\Yun_{u_0,s} (z)= \frac{\log (z)^{s-s_0}}{(s-s_0)!} \mbox{ for } s\in \{s_0,\ldots,S\}$$
and $\Yun_i (z)= 0$ for all other $i\in I$, namely $i = (u_0,s)$ with $s<s_0$, $i = (u,s) $ with $u\neq u_0$, or $i = u \in \{0,\ldots,\mu-1\}$. Then we have $\Yun \in \calW$ and
$$R(\Yun) (z)= \sum_{s=s_0} ^S P_{u_0,s,n} (z) \frac{\log (z)^{s-s_0}}{(s-s_0)!} .$$
For any $s\in \{s_0,\ldots,S\}$ Eq. \eqref{eqpetitdeg} yields $ P_{u_0,s,n} = O(z^{n+1-u_0})$ as $z\to 0$, so that
\begin{equation} \label{eqordiv}
\ord_0 R(\Yun)\geq n+1-u_0.
 \end{equation}
For any fixed $u_0 \in \{1,\ldots,m-1\}$ we have obtained $S$ vanishing conditions \eqref{eqordiv} (namely for $1\leq s_0 \leq S$), with non-holomorphic remainders; they correspond to the $S$ equations \eqref{eqpetitdeg} (for $1\leq s \leq S$) in which no logarithm appears, but no solution of the differential system $Y'=AY$ either. This is an illustration of the phenomenon explained in \S \ref{subsecnonholopade}: to apply Shidlovsky's lemma we may have to translate some vanishing conditions in order to express them in terms of solutions of the differential system. 

\bigskip

Let us deal now with assertion $(i)$ of Theorem \ref{thpade}. Recall that $\roro = \dim (\cz\cap\ker L)$, $\ell = \dim (\frac{\hz\cap\ker L}{\cz\cap\ker L}) $, and that in \S \ref{subsecshiddebut} we have chosen a basis $( f^{\{1\}},\ldots, f^{\{\roro\}})$ of $\cz\cap\ker L$. Let $f^{\{\roro+1\}},\ldots, f^{\{\roro+\ell\} }$ be such that $( f^{\{1\}},\ldots, f^{\{\roro+\ell\}})$ is a basis of $\hz\cap\ker L$. At last, choose $ f^{\{\roro+\ell+1\}},\ldots, f^{\{\mu\}} $ such that $( f^{\{1\}},\ldots, f^{\{\mu\}})$ is a basis of $\ker L$. For any $\ppp\in\unmu$ we let 
$$\Yde = \Big( ( (f^{\{\ppp\}})_u^{[s]}) _{1\leq u \leq \lun, 1\leq s \leq S}, \, (\theta^u f^{\{\ppp\}}) _{0\leq u \leq \mu-1} \Big),$$
which is consistent with the notation $\Yde$ introduced in \S \ref{subsecshiddebut} for $\ppp \in \{1, \ldots, \roro\}$. 
 Then we have $\Yde\in\calW$, and $R(\Yde)(z) = J_{f^{\{\ppp\}}}(z)$ for any $\ppp \in \{1,\ldots,\mu\}$. Using Eqns. \eqref{eqpadehol} and \eqref{eqpadenonhol} we obtain 
 \begin{equation} \label{eqordii}
\ord_0 R(\Yde)\geq \left\{ \begin{array}{l} (r+1)n +1 \mbox{ for any } \ppp \in\unrhoplusell , \\
n-\kappa \mbox{ for any } \ppp \in \{\roro+\ell+1,\ldots,\mu\}.
\end{array} \right.
 \end{equation}

 \bigskip

We have proved in \S \ref{subsecshiddebut} that 
 \begin{equation} \label{eqzerorho}
R(\Yde) \mbox{ is identically zero for }1\leq \ppp\leq \roro.
 \end{equation}
 This condition ``replaces'' the vanishing condition \eqref{eqordii} for these values of $\ppp$. 
 
 \bigskip
 
 In order to apply Theorem \ref{thshid} we still have to prove some results of linear independence. To begin with, as noticed in \S \ref{subsecshiddebut}, 
 the functions $ f^{\{1\}},\ldots, f^{\{\roro\}}$ are linearly independent over $\C$, so that the vectors $\Yde$ (for $1\leq \ppp \leq \roro$) involved in Eq. \eqref{eqzerorho} are also linearly independent (recall that $f^{\{\ppp\}} = \theta^0 f^{\{\ppp\}} $ is a component of $\Yde$). 

Let us prove now that the functions $R(\Yun)$ and $R(\Yde)$ involved in Eqns. \eqref{eqordiv} and \eqref{eqordii} (with $\roro+1\leq \ppp\leq\mu$) are linearly independent over $\C$. Let $\lamde$ (for $\roro+1\leq \ppp \leq \mu$) and $\lamun$ (for $1\leq u_0 \leq m-1$ and $1\leq s_0\leq S$) be complex numbers such that
$$ \sum_{\ppp=\roro+1}^\mu \lamde R(\Yde) + \sum_{u_0=1}^{m-1} \sum_{s_0 = 1}^S \lamun R(\Yun) =0.$$
Letting $f = \sum_{\ppp=\roro+1}^\mu \lamde f^{\{\ppp\}}\in\ker L$ we have using Lemma \ref{lemvraiesol}:
\begin{equation} \label{eqintermedde}
J_f(z) + \sum_{u_0=1}^{m-1} \sum_{s_0 = 1}^S \lamun \sum_{s=s_0} ^S P_{u_0,s,n}(z) \frac{\log (z)^{s-s_0}}{(s-s_0)!} = 0 .
 \end{equation}
Now $f $ belongs to the Nilsson class with rational exponents at 0; assume that $f\neq 0$. Then we may write 
$$f(z) = \sum_{k\in \Q \atop k\geq \capaf} \sum_{i=0}^{e} a_{k,i} z^k \log (z)^i $$
with $e$ bounded in terms of $L$ only. Therefore we may assume $S>e$, and Corollary \ref{corbourrin} yields $f\in\cz$ using Eq. \eqref{eqintermedde}. By construction of $f^{\{\roro+1\}}$, \ldots, $f^{\{\mu\}}$ this implies $f=0$. Since $f^{\{\roro+1\}} $, \ldots,$f^{\{\mu\}} $ are linearly independent over $\C$ we deduce that $ \lamde =0$ for any $\ppp\in\{ \roro+ 1, \ldots, \mu \}$. 

Moreover, recall from the proof of Lemma \ref{lem600} that $P_{u_0,s,n} (z) = c_{u_0,s,n} z^{n+1-u_0}$ for any $u_0\in \{1,\ldots,m-1\}$ and any $s\in \{1,\ldots,S\}$. Since $f=0$, Eq. \eqref{eqintermedde} reads
$$\sum_{u_0 =1}^{m-1} z^{n+1-u_0} \sum_{\sigma=0}^{S-1} \frac{\log (z)^\sigma}{\sigma!} \sum_{s_0 =1}^{S-\sigma} \lamun c_{u_0,\sigma+s_0, n} = 0$$
for any $z\in\dpr$, so that $ \sum_{s_0 =1}^{S-\sigma} \lamun c_{u_0,\sigma+s_0, n} = 0$ for any $u_0$ and any $\sigma$. With $\sigma=S-1$ we obtain $\lamunun = 0$ since $ c_{u_0,S, n} \neq 0$; by induction on $s_0$ it follows in the same way that $\lamun=0$ for any $u_0$ and any $s_0$. Finally, the functions $R(\Yde)$ and $R(\Yun)$ involved in Eqns. \eqref{eqordiv} and \eqref{eqordii} (with $\roro+1\leq \ppp \leq \mu$) are linearly independent over $\C$.

\bigskip

Combining these results of linear independence with Eqns. \eqref{eqannulalp}, \eqref{eqordiv}, \eqref{eqordii} and \eqref{eqzerorho},
 we have checked assumption $(ii)$ of Theorem \ref{thshid} with $\tau$ independent from $n$, since $\sum_{\alpha\in\Sigma\setminus\{0\}}m_\alpha = \ell$ (see \S \ref{subseccompte}) and $q = \lun S+\mu$. As explained at the beginning of \S \ref{subsecshidfin}, assertion $(ii)$ of Proposition \ref{propshid} follows.

\section{A Siegel-type linear independence criterion}\label{sec:critere}

The following criterion is based on Siegel's ideas (see for instance \cite[pp. 81--82 and 215--216]{EMS}, \cite[\S 3]{Matala-Aho}, \cite[\S 4.6]{SFcaract} or \cite[Proposition 4.1]{Marcovecchio}). 

\bigskip

Let $\K$ be a number field embedded in $\C$; denote by $\OK$ its ring of integers. We fix an embedding of $\K$ in a Galois closure $\L$, and of $\L$ in $\C$, so that Galois conjugates of elements of $\K$ can be seen as complex numbers. 
Given $\xi\in\K$, we denote by $\house{\xi}$ the house of $\xi$, i.e. the maximum modulus of the Galois conjugates of $\xi$.

\begin{Th}\label{th:siegel} 
Let $(Q_n)$ be an increasing sequence of positive real numbers, with limit $+\infty$.

Consider $N$ numbers $\vartheta_1, \ldots, \vartheta_N\in\C$. Assume that for some $\tau > 0$ there exist $N^2$ sequences $(p_{i,n}^{(j)})_{n\ge 0}$, $i,j=1, \ldots, N$, such that:
\begin{itemize}
\item For any $i$, $j$ and $n$, we have $p_{i,n}^{(j)}\in \OK$.
\item For any $i$, $j$,we have $\house{ p_{i,n}^{(j)} } \leq Q_n^{1+o(1)} $ as $n\to\infty$.
\item For any $j$ we have, as $n\to\infty$:
$$ \bigg| \sum_{i=1}^N p_{i,n}^{(j)}\, \vartheta_i \bigg| \leq Q_n^{-\tau+o(1)}.$$
\item For any $n$ sufficiently large the matrix $[ p_{i,n}^{(j)} ]_{1\leq i,j\leq N}$ is invertible.
\end{itemize}
Then 
$$
\dim_\K \textup{Span}_{\K} (\vartheta_1, \ldots, \vartheta_N) \geq \frac{\tau+1}{[\K:\Q]},
$$
and this lower bound can be refined to $\frac{2(\tau+1)}{[\K:\Q]}$ if $\K$ (seen as a subset of $\C$) is not contained in $\R$.
\end{Th}

Given real numbers $0 <a_0 < 1 <b$, this theorem can be applied when
 $\house{ p_{i,n}^{(j)} } \leq b^{n(1+o(1))}$ and 
$ \Big| \sum_{i=1}^N p_{i,n}^{(j)} \vartheta_i \Big| \leq a_0^{n(1+o(1))}$; then
$$
\dim_\K \textup{Span}_{\K} (\vartheta_1, \ldots, \vartheta_N)\ge \frac{1}{[\K:\Q]} \Big( 1 - \frac{\log(a_0)}{\log(b)}\Big) 
$$
with the right hand side multiplied by 2 if $\K\not\subset\R$.

\bigskip

\begin{proof}The proof is very classical, and similar (for instance) to that of \cite[Proposition 4.1]{Marcovecchio}. We sketch it for the convenience of the reader. Let $\Theta \in M_{N,1}( \C )$ denote the column matrix $\tra\, [\vartheta_1 \ldots \vartheta_N]$, and $\delta = \dim_\K \textup{Span}_{\K} (\vartheta_1, \ldots, \vartheta_N) $. There exists a matrix $A \in M_{N-\delta,N}(\OK)$ of rank $N-\delta$ such that $A\Theta=0$. Let $n\geq 0$. Since $P_n := [ p_{i,n}^{(j)} ]_{1\leq i,j\leq N}$ is invertible, we may assume (up to a permutation of the indices $j$) that the matrix
$$B := \left( \begin{array}{ccc}
 p_{1,n}^{(1)}& \ldots & p_{N,n}^{(1)} \\ 
\vdots & & \vdots \\
 p_{1,n}^{(\delta)}& \ldots & p_{N,n}^{(\delta)} \\ 
 \\
& A & \\ \mbox{ } 
\end{array}\right) \in M_N(\OK)$$
is invertible. Then $N_{\K/\Q}(\det B)$ is a non-zero rational integer; here $N_{\K/\Q}(x)=\prod_{\sigma} \sigma(x)$ is the norm of $x\in\K$, and $\sigma$ ranges through the set of embeddings $\K\to\C$. Accordingly $|N_{\K/\Q}(\det B)|\geq 1$. 

Now this positive integer can be bounded from above as follows. Denoting by $L_i$ the $i$-th column of $B$ and considering the embedding $\sigma = \Id$ (recall that $\K$ is seen as a subset of $\C$ from the beginning), the first $\delta$ coefficients of $\sum_{i=1}^N \vartheta_i L_i$ have modulus less than $ Q_n^{-\tau+o(1)}$, while the last $N-\delta$ coefficients are zero. Since $A$ does not depend on $n$, this leads to the following inequality as $n\to\infty$ (applying trivial bounds for $\sigma\neq \Id$):
$$1 \leq |N_{\K/\Q}(\det B)| \leq Q_n^{-\tau+o(1)} Q_n^{\delta-1 +o(1)} \prod_{\sigma\neq\Id} Q_n^{\delta +o(1)}.$$
Since $Q_n\to+\infty$ this implies $\delta \geq \frac{\tau+1}{[\K:\Q]}$. If $\K\not\subset\R$ the non-trivial upper bound on $\sigma(\det B)$ can be used not only for $\sigma=\Id$, but also for the complex conjugation; this yields $\delta \geq \frac{2(\tau+1)}{[\K:\Q]}$ and concludes the proof of Theorem \ref{th:siegel}.
\end{proof}

\section{Analytic and arithmetic estimates} \label{sec:estimates}

This section is devoted to technical estimates that will be used in \S \ref{secfin} to conclude the proof of Theorem \ref{thintroun}. In \S \ref{subsec:ubJF} we obtain an upper bound for $\big\vert J_F^{(k-1)}(z)\big\vert$. The important point is that we {\em do not} assume $z$ to be in the disk of convergence of the local expansion of $F$ at 0; we use analytic continuation and an integral representation. In \S \ref{subsecautres} we estimate the denominators and the size of the coefficients of the linear forms.

\subsection{An upper bound for $\big\vert J_F^{(k-1)}(z)\big\vert$} \label{subsec:ubJF}

The $G$-function $F$ of Theorem \ref{thintroun} can be analytically continued 
%outside its disk of convergence $\vert z\vert < R$ as follows. Consider a singularity $\alpha\in \Sigma_F$ of $F(z)$; it is either a pole or a branch point. As in the introduction we define $\Deltapr_\alpha:=\alpha +e^{i\arg(\alpha)} \mathbb R^+$, the straight line ``from $\alpha$ to $\infty$'' whose direction goes through 0 but $0\notin \Deltapr_{\alpha}$. Then the function $F$ can be analytically continued 
to the domain $\mathcal{D}_F$, which is star-shaped at $0$, as explained in the introduction. Recall from \S \ref{subsecnotationspade} that 
$$
J_F(z)=n!^{s-r}\sum_{k=0}^\infty \frac{k(k-1)\cdots (k-rn+1)}{(k+1)^S(k+2)^S\cdots (k+n+1)^S} A_k z^{k+n+1}
$$
for $|z|<R$, where $R$ is the radius of convergence of the local expansion $\sum_{k=0}^\infty A_k z^k$ of $F(z)$ around 0. 
By Proposition 3 in \cite{fns}, for any $z$ such that $\vert z\vert <R$, we have
\begin{equation}\label{eq:tan1}
J_F(z)=\frac{z^{(r+1)n+1}}{n!^r} \int_{[0,1]^{S}} F^{(rn)}(zt_1\cdots t_S)\prod_{j=1}^S t_j^{rn}(1-t_j)^{n} dt_1\cdots d t_S.
\end{equation}
Now, using the continuation of $F$ to $\mathcal{D}_F$, we see from \eqref{eq:tan1} that $J_F$ can be analytically continued to $\mathcal{D}_F$ as well; indeed for $z\in \mathcal{D}_F$ and $t\in[0,1]$, we have $zt\in\mathcal{D}_F$. Observe that $\dpr\subsetneq \mathcal{D}_F$ because the definition of $\dpr$ involves a half-line starting at 0, and possibly half-lines starting at singularities of $L$ at which $F$ is holomorphic; in the previous sections, $J_F$ was analytically continued to $\dpr$ only. 

\bigskip

We now fix an integer $k \geq 1$, that will be fixed even as $n\to\infty$. By \eqref{eq:tan1}, we have
\begin{multline} \label{eq:tan11}
J_F^{(k-1)}(z)= \sum_{i=0}^{k-1} \binom{k-1}{i}((r+1)n+i-k+3)_{k-i-1} \\
\times \frac{z^{(r+1)n+i-k+2}}{n!^r}\int_{[0,1]^{S}} F^{(rn+i)}(zt_1\cdots t_S)\prod_{j=1}^S t_j^{rn+i}(1-t_j)^{n} dt_1\cdots d t_S.
\end{multline}

Let us fix $z\in \mathcal{D}_F$. We can find a simple smooth direct contour $\mathcal{C}_z \subset \mathcal{D}_F$ such that for any $t\in[0,1]$, the segment $[0, zt]$ is at positive distance inside $\mathcal{C}_z$. By Cauchy formula,
$$
F^{(rn+i)}(zt)=\frac{(rn+i)!}{2i\pi}\int_{\mathcal{C}_z} \frac{F(x)}{(x-zt)^{rn+i+1}} dx.
$$
Since the functions 
$$
g(z):=\max\Big(1,\max_{x\in \mathcal{C}_z, t\in [0,1]} \frac{1}{\vert x-zt\vert}\Big)
$$
and $$h(z):=\frac{1}{2\pi}\textup{length}(\mathcal{C}_z)\max_{x\in \mathcal{C}_z}\vert F(x)\vert$$ 
are well defined and finite for any $z\in \mathcal{D}_F$, 
%we then deduce the existence of a function $g(z) : \mathcal{D}\to \mathbb R^+$, finite for any $z\in \mathcal{D}$, such that for any $t\in [0,1]$, 
%\begin{equation}\label{eq:tan2}
%\vert F^{(rn)}(zt) \vert \le (rn)! g(z)^{rn+1}.
%\end{equation}
we thus deduce for any $t\in[0,1]$ and any $0\le i \le k-1$:
\begin{equation}\label{eq:tan2}
\vert F^{(rn+i)}(zt) \vert \le (rn+k-1)! h(z) g(z)^{rn+k}.
\end{equation}

We can now give an upper bound for $\big\vert J_F^{(k-1)}(z)\big\vert$.

\begin{Prop}\label{prop:tan3} For any integers $S \ge r \ge 0$ and $k\geq 1$, and any $z\in \mathcal{D}_F$, we have
\begin{equation} \label{eq:tan4}
\limsup_{n\to +\infty} \big\vert J_F^{(k-1)}(z)\big\vert^{1/n} 
\le \frac{\max(1,\vert z\vert)^{r+1} g(z)^r}{(r+1)^{S-r}}.
\end{equation}
\end{Prop}
\begin{proof} In the end we shall make $n\to +\infty$ while keeping the other parameters fixed. We can thus assume that $n\ge k-1$ without loss of generality, so that $0\le (r+1)n+i-k+2\le (r+1)n+1$ for $0\le i\le k-1$. 
We set $\widetilde{z}$ for $\max(1, \vert z\vert)$.
We use \eqref{eq:tan2} in \eqref{eq:tan11} with $t=t_1t_2\cdots t_S$ and get 
\begin{align*}
\vert J_F^{(k-1)}(z) \vert
& \le 
\frac{(rn+k-1)!h(z) g(z)^{rn+k} \widetilde{z}^{(r+1)n+1}}{n!^r}
\\
& \qquad \times \sum_{i=0}^{k-1} 
\binom{k-1}{i}((r+1)n+i-k+3)_{k-i-1} 
 \int_{[0,1]^{S}} \prod_{j=1}^S t_j^{rn+i}(1-t_j)^{n} dt_1\cdots d t_S
\\
& \le k 2^{k-1} ((r+1)n+2)^{k-1} \frac{(rn+k-1)!h(z) g(z)^{rn+k} \widetilde{z}^{(r+1)n+1} }{n!^r}\left(\int_0^1 t^{rn}(1-t)^n dt \right)^{S}
\\
& = k 2^{k-1} ((r+1)n+2)^{k-1} h(z) g(z)^{rn+k} \widetilde{z}^{(r+1)n+1} \cdot \frac{(rn+k-1)!}{n!^r}
\cdot \left(\frac{n!(rn)!}{((r+1)n+1)!}\right)^S.
\end{align*}
Now by Stirling's formula (see \cite{JMPA} for a similar computation), we readily obtain 
$$
\limsup_{n\to +\infty} \big\vert J_F^{(k-1)}(z)\big\vert^{1/n} 
\le \widetilde{z}^{r+1} g(z)^r \frac{r^{r(S+1)}}{(r+1)^{S(r+1)}}
\le \frac{\widetilde{z}^{r+1} g(z)^r}{(r+1)^{S-r}},
$$
as expected.
\end{proof}

\subsection{Denominators and size of the coefficients} \label{subsecautres}

In this section we prove the last estimates to be used in the proof of Theorem \ref{thintroun}, namely those on the denominators and the size of the coefficients of the linear forms. As in \S \ref{sec:critere} we denote by $\OK$ the ring of integers of $\K$; we consider the rational functions $ \bfP_{k,i}\in\K(z)$ (see the beginning of \S \ref{secapplishid}). Recall that $ \bfP_{k,i}$ depends also on $n$, and that the set $\Sigma$ of finite singularities of $L$ contains all poles of the $ \bfP_{k,i}$.

\begin{lem}\label{lemarith}
 Let $z_0\in\K\setminus\Sigma$ and $v\in\N\etoile$ be such that $vz\in\OK$; let $K\geq 1$. Then there exists a sequence $(\delta_{n,K})_{n\geq 1}$ of positive rational integers 
 such that for any $i\in I$ and any $k\in \{1,\ldots,K\}$: 
$$
\delta_{n,K} \bfP_{k,i}(z_0) \in \OK \quad \textup{for any $n$, and} \quad 
\lim_{n\to +\infty} \delta_{n,K}^{1/n} = v C_2^S e^{S}.
$$
\end{lem}
\begin{proof} Let $d_n=\textup{lcm}\{1,2,\ldots,n\}$. As in \cite{fns}, the proof of~\cite[Lemme 5]{ribordeaux} shows that $d_n^S c_{j,s,n}\in \mathbb Z$ for all $j,s,n$; recall that $\lim_n d_n^{1/n}=e$. Upon multiplying $D(S,n)$ with a suitable positive integer, we may assume in 
Proposition~\ref{prop:1} $(iii)$ that $D(S,n)\geq C_2^{Sn}/2$, so that $\lim_n D(S,n)^{1/n} = C_2^{S}$. 
Moreover Proposition~\ref{prop:1} and Eqns. \eqref{eq:101} and \eqref{eq:102} yield
 $ d_n^S D(S,n+1) \bfP_i \in \OK[z]$ for any $i \in I$. Now let $T\in\OK[z]$ be such that $TA\in M_q(\OK[z])$, where $A\in M_q(\K(z))$ is the matrix of the differential system (see \S \ref{secapplishid}); we may assume that all roots of $T$ are poles of coefficients of $A$, so that $T(z_0)\neq 0$ since $z_0 \notin \Sigma$. Then Eq. \eqref{eqdefpki} yields $d_n^S D(S,n+1) T(z)^k \bfP_{k,i}(z) \in \OK[z]$ for any $i\in I$, by induction on $k\geq 1$. Since $\deg (T^k \bfP_{k,i}) \leq k \deg T + n+1+S(\ell-1)$ and $k\leq K$, we obtain $\delta'_{n,K} P_{k,i}(z_0) \in \OK$ by letting 
 $$
 \delta'_{n,K} = v^{ K \deg T + n+1+S(\ell-1)} d_n^S D(S,n+1) T(z_0)^K \in \OK.
$$

Now let $N_{\K/\Q}$ denote the norm relative to the extension $\K/\Q$, as in the proof of Theorem \ref{th:siegel} (see \S \ref{sec:critere}). Since $v^{K\deg T} T(z_0)^K \in \OK\setminus\{0\}$ we have $N_{\K/\Q} ( v^{K\deg T} T(z_0)^K) \in\N\etoile$ and $\frac{ N_{\K/\Q} ( v^{K\deg T} T(z_0)^K) }{ v^{K\deg T} T(z_0)^K} \in \OK$. Therefore letting 
$$
 \delta_{n,K} = v^{ n+1+S(\ell-1)} d_n^S D(S,n+1) \, N_{\K/\Q} ( v^{K\deg T} T(z_0)^K) \in \N\etoile 
$$
concludes the proof of Lemma \ref{lemarith}.
\end{proof}

Given $\xi\in\Qbar$, recall from \S \ref{sec:critere} that $\house{\xi}$ is the house of $\xi$, i.e. the maximum modulus of the Galois conjugates of $\xi$.

\begin{lem}\label{lemasy} Let $z_0\in\K\setminus\Sigma$ and $K\geq 1$. Then we have for any $i\in I$:
$$
\limsup_{n\to +\infty} \big(\max_{1\leq k \leq K }\house{ \bfP_{k,i}(z_0)}\big)^{1/n}\le C_1^Sr^r2^{S+r+1}\max(1, \house{z_0}).
$$
\end{lem}

\begin{Remark} The upper bounds in Lemmas \ref{lemarith} and \ref{lemasy} do not depend on $K$; they are the same as in the corresponding lemmas in \cite{fns}. Actually the important point in our application is that $K$ will be independent from $n$.
\end{Remark}

\begin{proof}
Given $P \in \K[z]$, we denote by $H(P)$ the maximum of the houses of its coefficients.
 In~\cite[Lemma 4]{ribordeaux}, it is proved that the coefficients $c_{j,s,n}$ in \eqref{eq:100} satisfy
$$
\vert c_{j,s,n}\vert \le (rn+1)2^S (r^r 2^{S+r+1})^n
$$
for all $j,s,n$. 
(Our $c_{j,s,n}$ are noted $c_{s,j-1,n}$ in \cite{ribordeaux}). 
Using this bound in \eqref{eq:101} and 
\eqref{eq:102} together with Proposition~\ref{prop:1}$(ii)$ we obtain $H( \bfP_i)\leq H_n$ for any $i$, with $\lim H_n^{1/n} = C_1^Sr^r2^{S+r+1}$. Now choose $T$ as in the proof of Lemma \ref{lemarith}, and let $\Pi_{k,i} = T^k \bfP_{k,i} \in \K[z]$ for any $k,i$. Then Eq. \eqref{eqdefpki} yields $\deg \Pi_{k,i} \leq c_A k+n$ by induction on $k$, and then $H(\Pi_{k+1,i})\leq (c'_Ak+n) H(\Pi_{k,i} )$, where $c_A$ and $c'_A$ depend only on $A$. For $k\leq K$ we deduce $\house{\Pi_{k,i} (z_0)}\leq (c_AK+n)H(\Pi_{k,i} ) \max(1, \house{z_0})^{c_AK+n}$ with $ H(\Pi_{k,i} ) \leq (c'_AK+n)^K H_n$. This enables us to conclude the proof of Lemma \ref{lemasy}.
\end{proof}

\section{Proof of Theorem \ref{thintroun}} \label{secfin}

In this section we prove Theorem \ref{thintroun} by combining the results obtained in previous sections.

\begin{proof}
Let $F(z) = \sum_{k=0}^\infty A_k z^k $ be a $G$-function, with $A_k \in\K$ and $F(z)\notin\cz$. As in the introduction we denote by 
$\Sigma_F$ the set of finite singularities of $F$
and for $\alpha\in \Sigma_F$ we define $\Deltapr_\alpha:=\alpha +e^{i\arg(\alpha)} \mathbb R^+$; we let 
 $\mathcal{D}_F:=\mathbb C \setminus (\cup_{\alpha \in \Sigma_F} \Deltapr_\alpha)$. Let $L_F$ and $z_0$ be as in Theorem \ref{thintroun}.

\bigskip

As in \S \ref{subsecnotationspade} we consider the $G$-operator $L$ obtained from $L_F$ by Lemma \ref{lemopdiff}, and define $\mu$, $\Sigma$, $\ell$ and $\dpr$ in terms of $L$. In defining $\dpr $ we choose a half-line $\Deltapr_0$ such that $z_0 \notin\Deltapr_0$, so that $z_0\in\dpr$. We also consider $\capazero$, $m$ and $\lun$ and in \S \ref{subseclienfns}, and integer parameters $S\geq r \geq 1$ with $S$ large enough in terms of $L$.

As in \S \ref{secapplishid} we let $$I = \Big( \{1,\ldots,\lun\}\times \{1,\ldots, S\} \Big) \, \, \sqcup \, \, \{0,\ldots,\mu-1\} $$
and $q = \Card \, I = \lun S+\mu $. Elements of $I$ are denoted by $(u,s)$ (with $1\leq u \leq \lun$ and $1\leq s \leq S$) or $u$ (with $0\leq u \leq \mu-1$). For any $n$ sufficiently large, Lemma \ref{lem600} provides a family $(\bfP_i)_{i\in I}$ of polynomials indexed by $I$, namely $\bfP_{u,s} = P_{u,s,n}$ and $\bfP_u = \Pti_{u,n}$; here the integer $n$ is omitted in the notation. 

\bigskip

For any $k\geq 1$ we are interested in the following linear form:
$$
J_F^{(k-1)}(z_0)=\sum_{u=1}^{\lun}\sum_{s=1}^S \bfP_{k,u,s}(z_0) F_{u}^{[s]}(z_0)+ \sum_{u=0}^{\mu-1} \bfP_{k,u}(z_0) (\theta^uF)(z_0) 
$$
obtained by taking the $(k-1)$-th derivative of the formula given by Lemma \ref{lem600}; here the rational functions $ \bfP_{k,i}$, for $i\in I$, are given by Eq. \eqref{eqdefpki} (see the beginning of \S \ref{secapplishid}). This formula can be written in a more compact way:
$$
J_F^{(k-1)}(z_0)=\sum_{i\in I} \bfP_{k,i}(z_0) y_i(z_0) 
$$
by letting $y_{u,s} = F_{u}^{[s]}$ and $y_u = \theta^uF$. Now Proposition \ref{propshid} provides elements $i_1$, \ldots, $i_\roro$ of $I$ and coefficients $\lambda_{i,t}\in\K$; assertion $(i)$ of this Proposition yields
\begin{equation} \label{eqJFun}
J_F^{(k-1)}(z_0)=\sum_{i\in I \setminus \{i_1,\ldots,i_\roro\}} \bfP_{k,i}(z_0) \Big( y_i(z_0) + \sum_{t=1}^\roro \lambda_{i,t} y_{i_t}(z_0) \Big).
\end{equation}

\bigskip

Let $N = q-\roro = \Card (I \setminus \{i_1,\ldots,i_\roro\})$, and choose a bijective map $\psi : I \setminus \{i_1,\ldots,i_\roro\} \to \{1,\ldots,N\}$. For any $i\in I \setminus \{i_1,\ldots,i_\roro\}$ and any $j\in \{1,\ldots,N\}$, let 
$$\vartheta_{\psi(i)} = y_i(z_0) + \sum_{t=1}^\roro \lambda_{i,t} y_{i_t}(z_0) 
\mbox{ and }
p_{\psi(i),n}^{(j)} = \delta_{n,K} \bfP_{k_j,i}(z_0) $$
where $k_1$, \ldots, $k_N$ are the integers provided by Proposition \ref{propshid} $(ii)$, $K$ is an upper bound on them, and $\delta_{n,K} $ is defined by Lemma \ref{lemarith}. The important point, here and below, is that $K$ depends only on $L$, $r$, $S$, but not on $n$ (eventhough the integers $k_1$, \ldots, $k_N$ depend on $n$). Then Eq. \eqref{eqJFun} yields
$$
J_F^{(k_j-1)}(z_0)=\sum_{i=1}^N p_{i,n}^{(j)} \, \vartheta_i \mbox{ for any } j\in \{1,\ldots,N\}.$$
Lemma \ref{lemarith} shows that all coefficients $p_{i,n}^{(j)} $ belong to $\OK$; Lemma \ref{lemasy} provides an upper bound on the moduli of their Galois conjugates. At last, Proposition \ref{propshid} asserts that the matrix $[ p_{i,n}^{(j)} ]_{1\leq i,j \leq N} $ is invertible for any $n$ sufficiently large. Therefore Siegel's linear independence criterion (namely Theorem \ref{th:siegel}) applied with $Q_n = b^n$ yields
\begin{equation} \label{eqminodim}
\dim_\K \textup{Span}_{\K} (\vartheta_1, \ldots, \vartheta_N) \geq \frac{\tau+1}{[\K:\Q]} 
\end{equation}
where 
$$\tau := \frac{- \log(a_0)}{\log(b)}, \quad
a_0 := v C_2^S e^{S} \frac{\max(1,\vert z\vert)^{r+1} g(z)^r}{(r+1)^{S-r}}, \quad
b:=v C_1^SC_2^Se^S r^r2^{S+r+1}\max(1, \house{z_0})$$
using Proposition \ref{prop:tan3} and Lemmas \ref{lemarith} and \ref{lemasy} (of which we keep the notation). 
Now for any $i$ the number $\vartheta_i$ belongs to the $\K$-vector space spanned by the numbers $F_{u}^{[s]}(z_0) $ and $(\theta^uF)(z_0) $, so that the lower bound \eqref{eqminodim} holds also with the dimension of this vector space in the left handside. We obtain therefore
\begin{equation} \label{eqminodimde}
\dim_\K \textup{Span}_{\K} \{ F_{u}^{[s]}(z_0) , \, 1\leq u \leq \lun, \, 1\leq s \leq S\} \geq \frac{\tau+1}{[\K:\Q]} - \mu.
\end{equation}
Taking for $r$ the integer part of $S/ \log(S)^2$, and letting $S$ tend to infinity, we deduce Theorem \ref{thintroun} with $C(F) = \log(2eC_1C_2)$. Observe that $C(F)$ depends only on $L$, and that this part of the computation is exactly the same as in \cite[\S 6.4]{fns}: $b$ is the same, and even though $a_0$ is slightly different the main term as $S\to\infty$ (with $r = \lfloor S/ \log (S)^2\rfloor$) is the same.
This concludes the proof of Theorem \ref{thintroun}.
\end{proof}

\bigskip

\noindent S. Fischler, 
Laboratoire de Math\'ematiques d'Orsay, Univ. Paris-Sud, CNRS, Universit\'e Paris-Saclay, 91405 Orsay, France.

\medskip

\noindent T. Rivoal, Institut Fourier, CNRS et Universit\'e Grenoble Alpes, CS 40700, 
 38058 Grenoble cedex 9, France

\bigskip

\noindent {\em Key words and phrases.} $G$-functions, $G$-operators, Pad\'e approximation, Siegel's linear independence criterion, Shidlovsky's lemma.

\medskip

\noindent {\em 2010 Mathematics Subject Classification.} Primary 11J72, 11J92, Secondary 34M35, 41A60.

\end{document}